\documentclass[12pt]{amsart}
\calclayout

\title[Statistical Bergman geometry]{Statistical Bergman geometry}
\author[Gunhee Cho]{Gunhee Cho}
\address{Department of Mathematics\\
	Texas State University\\
	601 University Drive, San Marcos, TX 78666.}
\email{wvx17@txstate.edu}
\author[J. Yum]{Jihun Yum}
\address{Department of Mathematics Education, Jeonbuk National University, 567, Baekje-daero, Deokjin-gu, Jeonju-si, Jeonbuk-do, 54896, Republic of Korea.}
\email{jihun0224@jbnu.ac.kr}
\date{\today}

\usepackage{amsthm,amssymb,latexsym}

\usepackage[abbrev]{amsrefs}
\usepackage{hyperref}

\usepackage[autostyle]{csquotes}
\usepackage{xcolor}
\usepackage{cite}

\usepackage{chemarrow}

\usepackage{kotex}
\usepackage{tikz-cd}

\newcommand{\RR}{\mathbb{R}}		
\newcommand{\CC}{\mathbb{C}} 		
\newcommand{\NN}{\mathbb{N}} 		
\newcommand{\PP}{\mathbb{P}} 		
\renewcommand{\O}{\Omega}			

\newcommand{\Prob}{\mathcal{P}}     
\newcommand{\Bergman}{\mathcal{B}}  
\newcommand{\Dia}{\mathcal{D}}      
\newcommand{\Poisson}{P}            
\newcommand{\Exp}{\mathrm{E}}
\newcommand{\Cov}{\mathrm{Cov}}
\newcommand{\Var}{\mathrm{Var}}

\newcommand{\norm}[1]{\left\|#1\right\|}                
\newcommand{\inner}[1]{\left\langle{#1}\right\rangle}   

\theoremstyle{plain}
\newtheorem{thm}{Theorem}[section]
\newtheorem{lem}[thm]{Lemma}
\newtheorem{prop}[thm]{Proposition}
\newtheorem{cor}[thm]{Corollary}
\newtheorem{ex}[thm]{Example}
\newtheorem{mainthm}{Theorem}

\theoremstyle{definition}
\newtheorem{defn}[thm]{Definition}

\theoremstyle{remark}
\newtheorem{rmk}[thm]{Remark}

\theoremstyle{property}

\numberwithin{equation}{section}

\begin{document}
	
	\maketitle

	\begin{abstract}
            This paper explores the Bergman geometry of bounded domains $\Omega$ in $\mathbb{C}^n$ through the lens of information geometry 
            by introducing a mapping $\Phi: \Omega \rightarrow \mathcal{P}(\Omega)$, where $\mathcal{P}(\Omega)$ denotes a space of probability measures on $\Omega$. 
            A result by J. Burbea and C. Rao establishes that the pullback of the Fisher information metric, the fundamental Riemannian pseudo-metric in information geometry, via $\Phi$ coincides with the Bergman metric of $\Omega$.
            Building on this idea, we consider $\Omega$ as a statistical model and present several interesting results within this framework.
            
            First, we derive a new statistical curvature formula for the Bergman metric by expressing it in terms of covariance. 
            Second, given a proper holomorphic map $f: \Omega_1 \rightarrow \Omega_2$, we prove that if the induced measure push-forward $\kappa: \mathcal{P}(\Omega_1) \rightarrow \mathcal{P}(\Omega_2)$ preserves the Fisher information metrics, then $f$ must be a biholomorphism.
            Finally, we establish the consistency and the central limit theorem of the Fr\'echet sample mean for Calabi's diastasis function.
	\end{abstract}

	\vspace{10mm}
	
	
	\section{Introduction}

    The Bergman kernel and Bergman metric, named after Stefan Bergman, are fundamental objects in the field of complex geometry and several complex variables. For a bounded domain $\O \subset \CC^n$, the Bergman kernel $\Bergman(z,\xi)$ is defined as the reproducing kernel of the space of $L^2$ holomorphic functions on $\O$. The Bergman metric $g_B$ is then defined as $g_B := i\partial \overline{\partial} \log \Bergman(z,z)$.
    One crucial property of the Bergman kernel and Bergman metric is their invariance under {{all biholomorphisms}.}
    S. Kobayashi's well-known result (\cite{Kobayashi59}) states that, by using $L^2$ holomorphic functions on $\O$, it is possible to map $\O$ to the infinite-dimensional complex projective space (called the Bergman-Bochner map), where the Bergman metric of $\O$ corresponds to the pullback of the Fubini-Study metric on the complex projective space.

    Let $\Prob(\O)$ denote a set of probability measures on $\O$. Information geometry studies a subset in $\Prob(\O)$ that can be parametrized by a finite-dimensional manifold $M$, called a statistical model (Definition~\ref{def: parametrized measure model}). Such a statistical model possesses a natural Riemannian pseudo-metric $g_F$, referred to as the Fisher information metric (Definition~\ref{def: Fisher information metric}). Using geometric structures such as Riemannian metrics, distances, connections, and curvatures, this field aims to enhance our understanding of the properties of $M$ and provide insights into the behavior of learning algorithms and optimization methods. As an example, consider the set $\mathcal{N}$ of all normal distributions on $\RR$. Each element in $\mathcal{N}$ can be uniquely determined by its mean $\mu \in (-\infty, \infty)$ and variance $\sigma^2 > 0$, which allows $\mathcal{N}$ to be parametrized by the upper-half plane in $\CC \cong \RR^2$. Surprisingly, $\mathcal{N}$ with the Fisher information metric is isometric to the upper-half plane with the Poincar\'e metric, up to a constant factor.

    The initial exploration of the relationship between the Fisher information metric and the Bergman metric is credited to J. Burbea and C. Rao in 1984 (\cite{BurbeaJacobRaoCRadhakrishna84}). They immersed a bounded domain $\O$ to a space of probability distributions and demonstrated that the Fisher information metric coincides with the Bergman metric of $\O$.
    Despite this remarkable discovery, further research in this direction has been limited.

    In this paper, we extend this line of investigation. The coincidence between the Bergman metric and the Fisher information metric implies that the Bergman metric can be expressed in terms of covariance (Definition~\ref{def: covariance}). Building on this idea, we derive a new statistical formula for the curvature of the Bergman metric by expressing it in terms of covariance as well.

        \begin{mainthm}[cf. Theorem~\ref{thm: statistical model with regular}, Theorem~\ref{thm: pullback of Fisher is Bergman} and Theorem~\ref{thm: statistical curvature formula}] \label{thm: A}
    	For a bounded domain $\O \subset \CC^n$, let $\Bergman(z, \xi)$ and $g_B(z) =  \sum_{\alpha, \beta} g_{\alpha \overline{\beta}}(z) dz_{\alpha} \otimes d\overline{z}_{\beta}$ be the Bergman kernel and the Bergman metric on $\O$, respectively. 
        Let $\Phi: \O \rightarrow \Prob(\O)$ be given by $\Phi(z) = P(z, \xi) dV(\xi) := \frac{|\Bergman(z, \xi)|^2}{\Bergman(z,z)} dV(\xi)$, where $dV$ is the Lebesgue measure on $\CC^n$. 
        Denote by $l = l(z,\xi) := \log P(z,\xi)$ and $\partial_{\alpha} := \frac{\partial}{\partial z_{\alpha}}$.
            Then
    	\begin{enumerate}
                \item 
                $\Phi: \O \rightarrow \Prob(\O)$ is a 2-integrable statistical model (Definition~\ref{def: parametrized measure model}) and 
                \begin{align*}
                    g_{\alpha \overline{\beta}} = \Cov\left[ \partial_{\alpha} l, \partial_{\beta} l \right],
                \end{align*}
                \item 
                the curvature 4-tensor $R_{\alpha \overline{\beta} \gamma \overline{\delta}}$ of $g_B$ is characterized by 
                \begin{align*}
                R_{\alpha \overline{\beta} \gamma \overline{\delta}}
                &= g_{\alpha \overline{\beta}} g_{\gamma \overline{\delta}} 
                + g_{\alpha \overline{\delta}} g_{\gamma \overline{\beta}}   \\
                &- \Cov\left[ \partial_{\alpha} \partial_{\gamma} l + (\partial_{\alpha} l)(\partial_{\gamma} l) -  \sum_{i=1}^n \Gamma_{\alpha \gamma}^i (\partial_i l), \,
                \partial_{\beta} \partial_{\delta} l + (\partial_{\beta} l)(\partial_{\delta} l) - \sum_{j=1}^n \Gamma_{\beta \delta}^j (\partial_j l) \right], \nonumber 
                \end{align*}
                where $\Gamma^{i}_{\alpha \gamma} = \sum_{\epsilon} g^{i \overline{\epsilon}}(\partial_{\alpha} g_{\gamma \overline{\epsilon}})$ is the Christoffel symbol for $g_B$, provided that $\O$ satisfies Condition $C_4$ (Definition~\ref{def: Condition C_k}).
                Moreover,
                \begin{align*}
                    N_{\alpha \gamma} := \partial_{\alpha} \partial_{\gamma} l + (\partial_{\alpha} l)(\partial_{\gamma} l) - \sum_{i=1}^n \Gamma_{\alpha \gamma}^i (\partial_i l) \in T_{\Phi(z)} \Prob(\O) 
                    \quad \text{ and } \quad 
                    N_{\alpha \gamma} \perp T_{\Phi(z)} \Phi(\O).
                \end{align*}
        	\end{enumerate}
        \end{mainthm}

    We refer to the statistical model $\Phi: \O \rightarrow \Prob(\O)$ as the \emph{Bergman statistical model}. \\

    Furthermore, we show that the map $\Phi: \O \rightarrow \Prob(\O)$ is essentially the same as the Bergman-Bochner map (Proposition~\ref{prop: essentially the Bergman-Bochner map}) and hence $\Phi$ is always an injective immersion (Theorem~\ref{thm: Phi is an injective immersion}).
    Also, we calculate various information--geometric structures (Theorem~\ref{thm: pullback of Information--geometric concepts}) for the Bergman statistical model such as the Amari-Chentsov tensor, the $\alpha$-connection, the Kullback–Leibler (KL) divergence and the entropy $H$, and prove that a Bergman statistical model is never an exponential family (Corollary~\ref{cor: not exponential family}).  
    It is interesting to note that the KL divergence for a Bergman statistical model becomes the Calabi's diastasis function.
    \\


    This new approach offers a novel perspective for reinterpreting Bergman geometry within the framework of information geometry and provides new insights into the geometric structure of complex manifolds through the lens of statistical models.
	One advantage of mapping a bounded domain into $\Prob(\O)$, compared to the Bergman-Bochner map, is that 
	for any measurable function $f: \O_1 \rightarrow \O_2$ between two bounded domains, there exists a natural map 
    $\kappa: \Prob(\O_1) \rightarrow \Prob(\O_2)$ induced by the measure push-forward via $f$.
	One of the most interesting features of the Fisher information metric is its monotonically decreasing property (Theorem \ref{thm: monotonicity Fisher metric}), i.e., $\kappa^* g_{F_2} \le g_{F_1}$.
	This implies that we obtain a new Riemannian metric $(\kappa \circ \Phi_1)^* g_{F_2}$ on $\O_1$ depending on $f$ which is always less than or equal to the Bergman metric of $\O_1$. However, it is important to note that this decreasing property alone does not directly imply the monotonically decreasing property of the Bergman metric (see Remark \ref{rmk: decreasing property of Bergman metric??}).

	When $\kappa$ is a local isometry (with respect to the Fisher information metrics), $f$ is called {\it sufficient} (Definition \ref{def: sufficient statistic}). 		
	Our second theorem says that a local isometry becomes a global isometry. 
	
	\begin{mainthm}[cf. Theorem \ref{thm: sufficient iff 1-1}] \label{thm: B}
		Assume that $f : \O_1 \rightarrow \O_2$ satisfies (\ref{f locally invertible}).
		Then $f$ is sufficient for the Bergman statistical model if and only if $f: f^{-1}(\O_2 \setminus Z) \rightarrow (\O_2 \setminus Z)$ is injective, where $Z := \{ f(z) \in \O_2 : |J_{\RR}f (z)|=0 \}$ is the set of all critical values.
	\end{mainthm}
    
	Theorem~\ref{thm: B} implies that a proper holomorphic map $f: \O_1 \rightarrow \O_2$ is a biholomorphism if and only if $f$ is sufficient for the Bergman statistical model (Corollary \ref{cor: proper holo 1-1 if and only if sufficient}). By applying Theorem~\ref{thm: B}, the second ordered author proved an interesting result that every local Bergman isometry is a biholomorphism as follows.  
    
        \begin{thm}[\cite{yum2026bergman}]
            For a proper holomorphic map $f\colon \Omega_1 \rightarrow \Omega_2$, if $f^* g_{B2} = \lambda g_{B1}$ holds on an open subset $U \subset \Omega_1$ for some constant $\lambda > 0$, then $f$ is a biholomorphism.\\
        \end{thm}


Let $(M, g)$ be a Riemannian manifold. Given a probability distribution $\mu_0$ on $M$, the {\it Fr\'echet mean} and the {\it Fr\'echet sample mean} (also known as the Karcher mean or the barycenter) of $\mu_0$ are defined as minimizers
\begin{align*}
    x_0 := \arg\min_{x \in M} \int_M d_g^2(x, y) d\mu_0(y)  \quad \text{ and } \quad 
    \hat{x}_m := \arg\min_{x \in M} \sum_{i=1}^{m} d_g^2(x, Y_i),
\end{align*}
respectively, where $d_g$ denotes the geodesic distance induced by $g$, and $Y_1, \cdots, Y_m$ are independent and identically distributed (i.i.d.) $M$-valued random variables with common distribution $\mu_0$. In Euclidean space, the Fr\'echet sample mean coincides with the usual arithmetic mean. 
Two fundamental asymptotic properties of the Fr\'echet sample mean $\hat{x}_m$ are consistency and the central limit theorem. 
Consistency asserts that $\hat{x}_m$ converges to the Fr\'echet mean $x_0$ in the sense that
$\lim_{m \to \infty} d_g(x_0, \hat{x}_m) = 0$ almost surely.
The central limit theorem states that $\sqrt{m} \exp_{x_0}^{-1}(\hat{x}_m)$ converges in distribution to a multivariate normal distribution on the tangent space $T_{x_0} M$ as $m \rightarrow \infty$.
These results do not hold in general; rather, they are established under certain geometric and probabilistic conditions. 
Characterizing the precise conditions under which consistency and central limit theorem of the Fr\'chet sample mean hold has been the subject of active research (  \cite{evans2024limit}, \cite{hotz2024central}, \cite{ziezold1977expected}).

E. Calabi (\cite{Calabi53}) introduced the so-called {\it diastasis function} for a real-analytic K\"ahler metric $g$ as a tool to characterize isometric embeddings of complex manifolds into $\CC^{\infty}$ or $\CC \PP^{\infty}$.
Although the diastasis function $\Dia$ is not a geodesic distance $d_g$, it shares many similarities with one. 
$\Dia$ is always non-negative and symmetric, and since $\Dia$ is a potential function of $g$, for a fixed point $p \in M$,
$$\Dia(x,p) = d_g(x,p)^2 + O(d_g(x,p)^4).$$
This approximation implies that the diastasis function serves as a natural alternative to the geodesic distance in many geometric contexts.
In fact, recent studies have explored replacing the squared geodesic distance with the diastasis function in the definition of various geometric invariants, such as diastatic entropy (\cite{bachas2014calabi}, \cite{mossa2014note}, \cite{mossa2016diastatic}) and the Laplace integral ( \cite{loi2005laplace}).

In this paper, on a bounded domain $\O$ with the Bergman metric $g_B$, we prove consistency and the central limit theorem of the Fr\'echet sample mean replacing the squared geodesic distance with the diastasis function $\Dia$ for $g_B$ (Definition~\ref{def: diastasis function}) as follows.

\begin{mainthm}[cf. Theorem~\ref{thm: argmin converges a.s.} and Theorem~\ref{thm: CTL for Bergman metric}] \label{thm: C}
    For a bounded domain $\Omega \subset \mathbb{C}^n$, fix $z_0 \in \Omega$. 
    Let $\{ Z_i \}_{i=1}^{\infty}$ be a sequence of i.i.d. random vectors drawn from the probability distribution 
    \[
    \Phi(z_0) = P(z_0, \xi) dV(\xi) = \frac{|\Bergman(z_0,\xi)|^2}{\Bergman(z_0,z_0)} dV(\xi).
    \]
    Suppose that $\O$ satisfies Condition $C_6$ (Definition~\ref{def: Condition C_k}). 
    Then, with probability tending to 1 as $m \rightarrow \infty$, there exists a sequence random vectors 
        \[
            \hat{z}_m := \arg \operatorname{loc} \min_{z \in \Omega} \sum_{i=1}^{m} \Dia(z, Z_i)
        \]
        such that
    \[
        \hat{z}_m \xrightarrow{p} z_0 \quad \text{ in probability}
    \]
    and 
    \[
    \sqrt{m} \left( \hat{z}_m - z_0 \right) \xrightarrow{d} \mathcal{N}_{\mathbb{C}}(0, g_B(z_0)^{-1},0) \quad \text{ in distribution}
    \]
    as $m \to \infty$, where the complex normal distribution $\mathcal{N}_{\CC}(0, g_B(z_0)^{-1},0)$ is given by equation (\ref{def: complex normal distribution when R=0}).
\end{mainthm}

\noindent
In particular, if $\Omega$ represents a bounded Hermitian symmetric domain, the above theorem follows directly (Corollary \ref{cor: CLT for bounded symmetric domain}).
It is interesting to note that the covariance of the normal distribution is the inverse of the Bergman metric at $z_0$.

The proof of Theorem~\ref{thm: C} relies on the asymptotic normality of the maximum likelihood estimator in the context of information geometry (\cite{LehmannELCasellaGeorge98}). 
One of the main challenges we faced arose from the fact that our sample space $\O$ is not compact, which differs from the standard assumptions required for central limit theorems and the asymptotic normality. 
Compactness of the sample space ensures the existence of the Fr\'echet sample mean and guarantees the finiteness of various expectations necessary for establishing results. 
To overcome this difficulty, we defined $\hat{z}_m$ as a local minimizer, and showed that for each fixed $m$, there exists a small neighborhood around $z_0$ in which such $\hat{z}_m$ exists (Theorem~\ref{thm: argmin converges a.s.}). 
Moreover, by assuming Condition $C_6$, we were able to ensure both the finiteness and the vanishing of certain expectation terms required in the proof. \\

	The article is organized as follows.
	In Section 2, we review basic terminology in Bergman geometry and information geometry.
	In Section 3, we define the map $\Phi: \O \rightarrow \Prob(\O)$ using the Bergman kernel and prove Theorem~\ref{thm: A}.
	In addition, we compute various information--geometric structures.
	In Section 4, we discuss the monotonically decreasing property of the Fisher information metric and prove Theorem~\ref{thm: B}. 
	Theorem~\ref{thm: C} is proved in section 5.
	For readers who may not be familiar with the terminology, we provide basic definitions and well-known theorems in information geometry and probability theory in Appendix A and B, respectively.

	\subsection*{The Acknowledgments}
	The authors would like to express their gratitude to Prof. Patrick Breheny for sharing valuable references, including \cite{LehmannELCasellaGeorge98}.
    This paper was finalized during the Midwest Several Complex Variables Conference, held to commemorate the mathematical contributions and retirement of Prof. Jeffery McNeal. The authors extend their thanks to Prof. Yum-Tong Siu, Prof. John D'Angelo, Prof. Emil Straube, Prof. L\'aszl\'o Lempert, Prof. Debraj Chakrabarti, and Prof. Andrew Zimmer for their insightful comments and words of encouragement.
    The authors would also like to acknowledge Hoseob Seo, SeungJae Lee, Thomas Pawlaschyk and Lucas Kaufmann for the fruitful discussions.
    
	The first ordered author is supported by the Simon Travel Grant, and the second ordered author is supported by the Institute for Basic Science (IBS-R032-D1).


\vspace{5mm}
    
    \section{Preliminaries}

    \subsection{Notation}
    We fix the notation of the paper.
    \begin{enumerate}
        \item[$\bullet$] $\O$: a bounded domain in $\CC^n$, \, $\Xi$: a measurable space,
        \item[$\bullet$] $dV$: the standard Lebesgue measure on $\CC^n$
        \item[$\bullet$] $A^2(\O)$: the space of all $L^2$ holomorphic functions on $\O$,
        \item[$\bullet$] $\Prob(\Xi)$: a space of probability measures on $\Xi$,
        \item[$\bullet$] $\Phi: \O \rightarrow \Prob(\Xi)$: a statistical model, \, $P(z, \cdot)$: the density function of $\Phi(z)$,
        \item[$\bullet$] $l = l(z,\xi) = \log P(z, \xi)$, \, $\partial_{\alpha} = \frac{\partial}{\partial z_{\alpha}}$, \, 
        $\partial_{\overline{\beta}} = \frac{\partial}{\partial \overline{z}_{\beta}}$,  \, $1 \le \alpha,\beta, \gamma, \delta \le n$,
        \item[$\bullet$] $\Bergman(z, \xi)$ : the Bergman kernel, \, $g_B$: the Bergman metric, 
        \item[$\bullet$] $g_F$ : the Fisher information metric,
        \item[$\bullet$] $\norm{f}_{L^m(\Xi, \mu)} = \left( \int_{\Xi} |f(\xi)|^m d\mu(\xi) \right)^{\frac{1}{m}}$,
        \item[$\bullet$] $\Dia$: the Calabi's diastasis function. 
    \end{enumerate}

\vspace{2mm}  
    
\subsection{Statistical model} \label{sec: statistical model}
We mostly follow definitions in \cite{ay2015information} and \cite{ay2017informationbook}. A brief review of information geometry is included in Appendix A (see also \cite{amari2000methods}).

Let $(\Xi, \Sigma)$ be a measurable space. 
We consider the Banach space, for a $\sigma$-finite non-negative (reference) measure $\mu_0$,
\begin{equation*}
	\mathcal{S}(\Xi, \mu_0) := \left\{ \mu = \phi \mu_0 : \phi \in L^1(\Xi, \mu_0) \right\}
\end{equation*}
with the total variation $\norm{\cdot}_{TV}$ as Banach norm. 
This space can be identified via the Banach space isomorphism 
$i_{can} : \mathcal{S}(\Xi, \mu_0) \rightarrow L^1(\Xi, \mu_0)$ defined by $\mu \mapsto \frac{d\mu}{d\mu_0}$, where $\frac{d\mu}{d\mu_0}$ is the Radon-Nikodym derivative.
We also consider the following subsets of $\mathcal{S}(\Xi, \mu_0)$:
\begin{align*}
	\mathcal{M}(\Xi, \mu_0) &:= \left\{ \mu = \phi \mu_0 \;\middle\vert\; \phi \in L^1(\Xi, \mu_0), \; \phi > 0 \; \mu_0 \text{-a.e.} \right\}, \\
	\Prob(\Xi, \mu_0) &:= \left\{ \mu = \phi \mu_0 \;\middle\vert\; \phi \in L^1(\Xi, \mu_0), \; \phi > 0 \; \mu_0 \text{-a.e.}, \; \int_{\Xi} d\mu = 1 \right\}.
\end{align*}
Two elements $\mu_1$ and $\mu_2$ are said to be equal if $\mu_1 = \mu_2$ $\mu_0$-a.e..
An element $\mu$ in $\Prob(\Xi, \mu_0)$ is called a {\it probability measure} on $\Xi$, and $\phi$ is called a {\it probability density function} on $\Xi$.
For the geometry of an infinite-dimensional space of probability measures, we refer to \cite{pistone1995infinite}. \\

Roughly speaking, a statistical model is a finite dimensional subset in $\Prob(\Xi, \mu_0)$. The precise definition of a statistical model may vary depending on additional technical assumptions. In this paper, we adopt the following definition.

\begin{defn}[cf. {\cite[Definition 3.4]{ay2017informationbook}}] \label{def: parametrized measure model}
	\leavevmode
	\begin{enumerate}
		\item A {\it parametrized measure model} is a quadruple $(M, \Xi, \mu_0, \Phi)$ consisting of a smooth manifold $M$, a measurable space $\Xi$, a $\sigma$-finite non-negative measure $\mu_0$ and a  Fr\'echet $C^1$-smooth map
		$$ \Phi: M \rightarrow  \mathcal{M}(\Xi, \mu_0) \subset \mathcal{S}(\Xi, \mu_0).$$
		\item A parametrized measure model $(M, \Xi, \mu_0, \Phi)$ is called a {\it statistical model} if $\Phi(M) \subset \Prob(\Xi, \mu_0)$. The model $(M, \Xi, \mu_0, \Phi)$ is also represented by $\Phi: M \rightarrow \Prob(\Xi, \mu_0)$.
		\item We say that a parametrized measure model $(M, \Xi, \mu_0, \Phi)$ has a {\it regular density function} if for the {\it density function} $P : M \times \Xi \rightarrow \RR$ given by 
		$$  P(x, \xi) := \frac{d \{\Phi(x)\}}{d\mu_0} \in L^1(\Xi, \mu_0),  $$
		the partial derivative $\partial_V P(x, \cdot)$ exists and lies in $L^1(\Xi, \mu_0)$ for all $V \in T_xM$. 
	\end{enumerate}
\end{defn}
We call $M$ a {\it parameter space} and $\Xi$ a {\it sample space}.\\

For a parametrized measure model $(M, \Xi, \mu_0, \Phi)$, denote $D \Phi(x)$ be the Fr\'echet derivative of $\Phi(x)$. Then for a tangent vector $V \in T_xM$, the {\it logarithmic derivative} of $\Phi$ at $x \in M$ is defined by 
\begin{align} \label{equ: dlog Phi}
	\partial_V \log \Phi(x) := \frac{d \{ D \Phi(V) \} }{d \Phi(x)} \in L^1(\Xi, \Phi(x)).
\end{align}
If a parametrized measure model has a positive regular density function $P$, then
$$ \partial_V \log \Phi(x) = \partial_V \log P(x, \cdot),
$$
which is independent of the choice of the (reference) measure $\mu_0$.
In this paper, by slight abuse of notation, we denote $\partial_V \log \Phi$ by $\partial_V \log P$, even for a model without a positive regular density function.

\begin{defn}[cf. {\cite[Definition 2.4]{ay2015information}}, {\cite[Theorem 3.2]{ay2017informationbook}}] \label{def: k-integrable}
	A parametrized measure model $(M, \Xi, \mu_0, \Phi)$ is called {\it $k$-integrable} ($k \ge 1$) if the map
	$$ V \longmapsto \norm{\partial_{V} \log P(x,\cdot)}_{L^k(\Xi, \Phi(x))} $$
	defined on $TM$ is continuous.
\end{defn}

\begin{rmk} \label{rmk: integral dP = 0}
	Every parametrized measure model is $1$-integrable by definition. Moreover, $k$-integrability implies $m$-integrability for $1 \le m < k$.
    The role for $k$-integrability is the interchangeability of differentiation and integration.
	For example, every $k$-integrable statistical model satisfies
	\begin{align*}
		\int_{\Xi} \partial_V P(x,\xi)d\mu_0(\xi) = \partial_V \int_{\Xi} P(x,\xi)d\mu_0(\xi) = 0
	\end{align*}
	(see \cite[Remark 2.5]{ay2015information}). 
	In other words, $\Exp[\partial_V \log P] = \int_{\Xi} (\partial_V \log P(x,\xi)) P(x,\xi)d\mu_0(\xi) = 0$.
\end{rmk}

The {\it tangent cone} of $\Prob(\Xi) (= \Prob(\Xi, \mu_0))$ at $\mu \in \Prob(\Xi)$ is given by (\cite[Theorem 3.1]{ay2017informationbook})
$$ T_{\mu} \Prob(\Xi) = \left\{ \sigma = \phi \mu  \; \middle\vert \; \phi \in L^1(\Xi, \mu),  \; \int_{\Xi} d\sigma = 0  \right\}. $$
The canonical isomorphism $\i_{\text{can}}: \sigma \mapsto \frac{d \sigma}{d \mu}$ induced by the Radon-Nikodym derivative yields
\begin{align*}
	T_{\mu} \Prob(\Xi) \cong \left\{  \phi \in L^1(\Xi, \mu) \; \middle\vert \; 
	\Exp[\phi] :=  \int_{\Xi} \phi d\mu = 0   \right\}.
\end{align*}

\begin{defn} \label{def: Fisher information metric}
	The {\it Fisher information metric} $g_F$ at $\mu \in \Prob(\Xi)$ is defined by
	\begin{equation}\label{eq:Fisher information-metric}
		(g_F)_{\mu}(\sigma_1, \sigma_2) := \int_{\Xi} \frac{d \sigma_1}{d \mu} \frac{d \sigma_2}{d \mu} d\mu,
	\end{equation}
	for $\sigma_1, \sigma_2 \in T_{\mu} \Prob(\Xi)$,
	where $\frac{d \sigma_1}{d \mu}$ and  $\frac{d \sigma_2}{d \mu}$ are the Radon-Nikodym derivatives.
\end{defn} 

It is interesting to note that T. Friedrich (\cite{friedrich1991fisher}) proved that the sectional curvature of $g_F$ on $\Prob(\Xi)$ is $\frac{1}{4}$ whenever they are well-defined.
Various geometric properties of Fisher information metric are given in \cite{ay2015information}. \\

A statistical model $(M, \Xi, \mu_0, \Phi)$ is called {\it immersed} if the map 
$d_x \log P : T_xM \rightarrow L^1(\Xi, \Phi(x))$ given by \eqref{equ: dlog Phi} is injective for all $x \in M$.
In this case, the tangent space of $\Phi(M)$ at $\mu = \Phi(x)$ can be identified by
\begin{align*}
	T_{\mu} (\Phi(M)) 
	\cong \langle \partial_{1} \log P(x, \cdot), \cdots, \partial_{n} \log P(x, \cdot)   \rangle,
\end{align*}
where $\partial_{\alpha} = \frac{\partial}{\partial x_{\alpha}}$ $(\alpha = 1, \cdots, n)$ and $\dim_{\RR} M = n$.
Moreover, the Fisher information metric $(\Phi^* g_F)(x) = \sum_{\alpha,\beta=1}^n g_{\alpha\beta}(x) dx_{\alpha} \otimes dx_{\beta}$ is given by 
\begin{equation*}
	g_{\alpha\beta}(x) = \int_{\Xi} (\partial_{\alpha} \log P(x,\xi))(\partial_{\beta} \log P(x,\xi)) P(x, \xi) d\mu_0(\xi)
	= \Cov[\partial_{\alpha} l, \partial_{\beta} l],
\end{equation*}
where $l = l(x,\xi) := \log P(x,\xi)$.
The Fisher information metric is well-defined (i.e., finite) for any $2$-integrable statistical model.
\vspace{2mm}
	
	\subsection{Hermitian metric}
	
	Let $M$ be a complex manifold with the complex structure $J: TM \rightarrow TM$, where $TM$ is the real tangent bundle.
	Since $J^2 = - \text{id}$, the complexified tangent space $T^{\CC}_p M := \CC \otimes T_pM$ at $p \in M$ can be decomposed into two eigenspaces
	$T^{(1,0)}_pM := \{ v \in  T^{\CC}_pM \,|\, Jv = iv \}$ and $T^{(0,1)}_pM := \{ v \in T^{\CC}_pM \,|\, Jv = -iv \}$. There is an one-to-one correspondence between $T_pM$ and $T^{(1,0)}_pM$ via 
	$$v \mapsto v^{\sim} := \frac{1}{2}(v - iJv). $$
	Note that $\overline{v^{\sim}} =  \frac{1}{2}(v + iJv) \in T^{(0,1)}_pM$.
	
	A Riemannian metric $g$ on $M$ is called {\it Hermitian} if $g(v,w) = g(Jv, Jw)$ for all $v,w \in T_pM$.
	Define $\hat{g}: T^{\CC} M  \times T^{\CC} M  \rightarrow \CC$ by extending $g$ $\CC$-linearly.
	Then $g$ is Hermitian if and only if $\hat{g}(v^{\sim}, w^{\sim}) = 0$ for all $v^{\sim}, w^{\sim} \in T^{(1,0)}_pM$. 
	
	From a Riemannian Hermitian metric $g$, one can induce a complex Hermitian metric $h : T^{(1,0)}M \times T^{(0,1)}M \rightarrow \CC$ by $h(v^{\sim}, \overline{w^{\sim}}) := \hat{g}(v^{\sim}, \overline{w^{\sim}})$.
	Conversely, a complex Hermitian metric $h$ induces a Riemannian Hermitian metric $g := \text{Re} h. $
	Note that for a Riemannian metric $g$ (not necessarily Hermitian), 
	$$ \hat{g} (v^{\sim}, v^{\sim}) = \frac{1}{4} \left( g(v,v) + g(Jv, Jv) \right) .$$

\vspace{2mm}  
    
\subsection{Bergman geometry} \label{sec: Bergman geometry}

        Let $\O \subset \CC^n$ be a bounded domain and $dV$ be the Lebesgue measure on $\CC^n$.
    	Let $A^2(\O)$ be the space of all square integrable holomorphic functions on $\O$ equipped with the Hermitian inner product 
    	$\inner{f,g} =  \int_{\O} f(z) \overline{g(z)} dV(z)$. Then $A^2(\O)$ becomes a Hilbert space.
    	The {\it Bergman kernel} function on $\O \times \O$ is defined by 
    	$$ \Bergman (z,\xi) := \sum_{j= 0}^{\infty} s_{j}(z)  \overline{s_{j}(\xi)} ,$$
    	where $\{ s_{j} \}_{j=0}^{\infty}$ is an orthonormal basis for $A^2(\O)$. 
    	Note that it is independent of the choice of an orthonormal basis. 
    	By the definition, $\Bergman(z,\xi)$ is holomorphic in the first variable $z$ and anti-holomorphic in the second variable $\xi$.
    	One of the most important features is the following reproducing property, which is induced from the Riesz representation theorem.
    	\begin{prop} \label{prop: Bergman reproducing}
    		For $f \in A^2(\O)$, we have
    		$$ f(z) = \int_{\O} \Bergman(z,\xi) f(\xi) dV(\xi) $$
    		for all $z \in \O$.
    	\end{prop}
    	
    	The {\it Bergman metric} $g_B$ on $\O$ is defined by
    	\begin{equation*}
    		g_B(z) := \sum_{\alpha,\beta=1}^n \frac{\partial^2 \log \Bergman(z,z)}{\partial z_{\alpha} \partial \overline{z}_{\beta} } dz_{\alpha} \otimes d\overline{z}_{\beta}
    	\end{equation*}
    	provided that $\Bergman(z,z) > 0$ for all $z \in \O$.
            From the definition, we know that $g_B$ is a K\"ahler metric.
    	On a bounded domain $\O \subset \CC^n$, it is well-known that $g_B$ is always well-defined and positive-definite. 
    
        \begin{defn} \label{def: diastasis function}
            The diastasis function for $g_B$ is given by
            \begin{equation*}
                \Dia(z,\xi) := \log \frac{\Bergman(z,z) \Bergman(\xi,\xi)}{|\Bergman(z,\xi)|^2}  .
            \end{equation*}
        \end{defn}
        \noindent
	This function plays a central role throughout the paper.

\vspace{2mm}  
\subsection{Special orthonormal basis} \label{subsec: Special orthonormal basis}

        In this subsection, we discuss the construction of a special orthonormal basis in details. 

        Let $\{ D^{k} \}_{k=1}^{\infty}$ denote a sequence of partial derivatives in lexicographic order, i.e.,
		$
            D^1 = \frac{\partial}{\partial z_1}, 
            D^2 = \frac{\partial}{\partial z_2}, \cdots, 
		D^n = \frac{\partial}{\partial z_n}, 
            D^{n+1} = \frac{\partial^2}{\partial z_1 \partial z_1},  
		D^{n+2} = \frac{\partial^2}{\partial z_1 \partial z_2}, \cdots $ and so on. 
        Define linear functionals $T^{0}_z: f \mapsto f(z)$ and
        $T^{k}_z: f \mapsto D^{k}f(z)$ for all $k \ge 1$.
        Then, Cauchy estimates for holomorphic functions imply that $T_z^k$ is bounded and $T_z^k$ belongs to the space $A^2(\O)^*$ of all bounded linear functionals. 
        Denote by $\norm{\cdot}_{L^2} := \norm{\cdot}_{L^2(\O, dV)}$.

        Now fix $z \in \O$ and donote by $H^0_z := \operatorname{Ker}(T_z^0)$.
        Since $T_z^0$ is bounded, the orthogonal complement $(H_z^0)^{\perp}$ in $A^2(\O)$ is a 1-dimensional subspace, and one can choose the unique element $s^z_0 \in A^2(\O)$ such that
        \begin{align*}
            \langle s^z_0, H^0_z \rangle = 0, \quad \norm{s^z_0}_{L^2}=1 \, \text{ and } \, s^z_0(z) >0.
        \end{align*}
        Inductively, for $k \ge 1$, let $H^k_z := \operatorname{Ker}(T_z^k|_{H^{k-1}_z})$ and choose the unique element $s^z_k \in H^{k-1}_z$ such that
        \begin{align*}
            \langle s^z_k, H^k_z \rangle = 0, \quad \norm{s^z_k}_{L^2}=1 \, \text{ and } \, D^k s^z_k(z) >0.
        \end{align*}
        Note that the boundedness of $\O$ guarantees that $T_z^k|_{H^{k-1}_z}$ is not the zero functional.
        Therefore, from the construction, one has
        \begin{align*}
            &s^z_0(z) \neq 0, \, s^z_1(z) = 0, \, s^z_2(z) = 0, \, s^z_3(z) = 0, \cdots, \\
            &D^1 s^z_1(z) \neq 0, \, D^1 s^z_2(z) = 0, \, D^1 s^z_3(z) = 0, \cdots, \\
            &D^2 s^z_2(z) \neq 0, \, D^2 s^z_3(z) = 0, \, D^2 s^z_4(z) = 0,  \cdots , \\
            &\hspace{35mm}  \vdots
        \end{align*}
        We call $\{ s^{z}_k \}_{k=0}^{\infty}$ a {\it special orthonormal basis} for $A^2(\O)$ with respect to $z \in \O$.
        By modifying the above argument, we may construct a special orthonormal basis $\{ s^{z}_k \}_{k=0}^{\infty}$ with respect to $z \in \O$ and $V \in T^{1,0}_z\O$ satisfying
        \begin{align*}
            &s^z_0(z) \neq 0, \, s^z_1(z) = 0, \, s^z_2(z) = 0, \, s^z_3(z) = 0, \cdots, \\
            &\partial_V s^z_1(z) \neq 0, \, \partial_V s^z_2(z) = 0, \, \partial_V s^z_3(z) = 0, \cdots.
        \end{align*}
        It is very useful to use a special orthonormal basis when we calculate Bergman invariants such as Bergman kernel and Bergman metric:
        \begin{align*}
            \Bergman(z,z) =  \sum_{j=0}^\infty |s^z_j(z)|^2 = |s^z_0(z)|^2
        \end{align*}
        and 
        \begin{align*}
            (g_B)_z(V, \overline{V}) 
            = &\frac{\partial_V \partial_{\overline{V}} \Bergman(z,z)}{\Bergman(z,z)} - \frac{\partial_V \Bergman(z,z) \partial_{\overline{V}} \Bergman(z,z)}{\Bergman(z,z)^2} \\
            = & \frac{|\partial_V s^z_0(z)|^2 + |\partial_V s^z_1(z)|^2  }{|s^z_0(z)|^2} - \frac{(\partial_V s^z_0(z) \overline{s^z_0(z)})(s^z_0(z) \overline{\partial_V s^z_0(z)})}{|s^z_0(z)|^4} \\
            = & \frac{|\partial_V s^z_1(z)|^2}{|s^z_0(z)|^2}.
        \end{align*}
        Here, $\partial_V s^z_1(z) := \sum_{j=1}^{n} v_j \left. \frac{\partial}{\partial \xi_j}\right|_{\xi=z} s^z_1(\xi)$ for $V = \sum_{j=1}^{n} v_j \frac{\partial}{\partial \xi_j}$. \\

        Note that a special orthonormal basis $\{ s^{z}_k \}_{k=0}^{\infty}$ varies as $z$ moves. 
        In the remaining of this subsection, we prove that $s^{z}_k$ is continuous in $z \in \O$ for all $k \ge 0$, which we will use later.  

        \begin{lem} \label{lem: T_z continuous}
            For each $k \ge 0$, the map $\O \rightarrow A^2(\O)^*$, $ z \mapsto T_z^k $, 
            is continuous, i.e.,
            $$  \lim_{z \rightarrow p} \norm{T_z^k - T_p^k}_{*} = 0 ,$$
            where $\norm{T}_* := \sup_{f \in A^2(\O)} \frac{|T(f)|}{\norm{f}_{L^2}}$ for $T \in A^2(\O)^*$.
        \end{lem}
        \begin{proof}
            For fixed $p \in \O$,
            \begin{align} \label{(3.3)}
                \norm{T_z^k - T_p^k}_{*} 
                &= \sup_{\norm{f}_{L^2}=1} |T_z^k(f) - T_p^k(f)| \\
                &= \sup_{\norm{f}_{L^2}=1} |D^k f(z) - D^k f(p)|. \nonumber
            \end{align}
            Since $D^k f(z)$ is a holomorphic function, 
            \begin{align*}
                |D^k f(z) - D^k f(p)| \le  \max_{1 \le \alpha \le n} \sup_{w \in \overline{U}_p} |\partial_{\alpha} D^k f(w)| |z-p|
            \end{align*}
            for all $z \in U_p$, where $U_p$ is a small neighborhood of $p$. Also, Cauchy estimates implies that 
            \begin{align*}
                \max_{1 \le \alpha \le n} \sup_{w \in \overline{U}_p} |\partial_{\alpha} D^k f(w)| \le C \norm{f}_{L^2(\O, dV)},
            \end{align*}
            where the constant $C>0$ does not depend on $f$. Altogether, from $(\ref{(3.3)})$,
            \begin{align*}
                \norm{T_z^k - T_p^k}_{*} \le \sup_{\norm{f}_{L^2}=1} C \norm{f}_{L^2} |z-p| = C |z-p| \rightarrow 0
            \end{align*}
            as $z \rightarrow p$.
        \end{proof}

        \begin{prop} \label{prop: s^z continuous}
            For a bounded domain $\O \subset \CC^n$,
            let $\{ s^{z}_k \}_{k=0}^{\infty}$ be a special orthonormal basis for $A^2(\O)$ with respect to $z \in \O$. Then, for each $k \ge 0$,
            \begin{enumerate}
                \item[(1)] The map $\O \rightarrow A^2(\O)$, $z \mapsto s^z_k$, is continuous, i.e.,
                $$ \lim_{z \rightarrow p} \norm{s^z_k - s^p_k}_{L^2}= 0 .$$
                \item[(2)] The map $\O \rightarrow \CC$, $z \mapsto D^l s^z_k(z)$, is continuous, i.e.,
                $$ \lim_{z \rightarrow p} |D^l s^z_k(z) - D^l s^p_k(p)|= 0 $$
                for all $l \ge 0$.
            \end{enumerate}
        \end{prop}
        \begin{proof}
            \begin{enumerate}
                \item[(1)] 
                We use the mathematical induction in $k$. Let $k=0$.
                The Riesz Representation theorem implies that there exists the unique element $h^z_0 \in A^2(\O)$ such that 
                $$ T_z^0(f) = \langle f, h^z_0 \rangle_{L^2} \, \forall f \in A^2(\O) \, \text{ with } \, \norm{T_z^0}_* = \norm{h^z_0}_{L^2} .$$
                Moreover,  
                \begin{align*}
                    \norm{h^z_0 - h^p_0}_{L^2} = \norm{T_z^0 - T_p^0}_* \rightarrow 0
                \end{align*}
                as $z \rightarrow p$ by Lemma~\ref{lem: T_z continuous}.
                On the other hands, 
                \begin{align*}
                    h_0^z = a_0^z s_0^z \, \text{ where } \overline{a_0^z} = \langle s_0^z, h^z_{0}  \rangle = s_0^z(z) > 0,
                \end{align*}
                and hence
                \begin{align*}
                    \lim_{z \rightarrow p} a_0^z = \lim_{z \rightarrow p} |a_0^z| = \lim_{z \rightarrow p} \norm{h_0^z}_{L^2} = \norm{h_0^p}_{L^2} = a_0^p.
                \end{align*}
                Altogether, 
                \begin{align*}
                    \lim_{z \rightarrow p} s_0^z = \lim_{z \rightarrow p} \left( \frac{h_0^z}{a_0^z} \right) = s_0^p.
                \end{align*}
                Now, for $N \ge 0$, suppose that $\lim_{z \rightarrow p} \norm{s^z_k - s^p_k}_{L^2}= 0$ for all $0 \le k \le N$.
                The Riesz Representation theorem implies that there exists the unique element $h^z_{N+1} \in A^2(\O)$ such that 
                $$ T_z^{N+1}(f) = \langle f, h^z_{N+1} \rangle_{L^2} \, \forall f \in A^2(\O) \, \text{ with } \, \norm{T_z^{N+1}}_* = \norm{h^z_{N+1}}_{L^2} .$$
                Moreover, 
                \begin{align} \label{33.4}
                    \norm{h^z_{N+1} - h^p_{N+1}}_{L^2} = \norm{T_z^{N+1} - T_p^{N+1}}_* \rightarrow 0
                \end{align}
                as $z \rightarrow p$ by Lemma~\ref{lem: T_z continuous}.
                On the other hands, 
                \begin{align*}
                    h^z_{N+1} = \sum_{j=0}^{N+1} a_j^z s_j^z, \, \text{ where } \overline{a_j^z} = \langle s_j^z, h^z_{N+1}  \rangle.
                \end{align*}
                Now, from the inductive hypothesis, 
                \begin{align} \label{33.5}
                    \lim_{z \rightarrow p} s_j^z = s_j^p 
                \end{align}
                and
                \begin{align} \label{33.6}
                    \lim_{z \rightarrow p} a_j^z = \langle \lim_{z \rightarrow p} h^z_{N+1}  , \lim_{z \rightarrow p} s_j^z \rangle = \langle h^p_{N+1}  ,s_j^p \rangle = a_j^p
                \end{align}
                for all $0 \le j \le N$.
                Also, 
                \begin{align*}
                    \lim_{z \rightarrow p} |a^z_{N+1}|^2 
                    = \lim_{z \rightarrow p} \left( \norm{h^z_{N+1}}^2_{L^2} - \sum_{j=0}^{N} |a_j^z|^2  \right)
                    = |a^p_{N+1}|^2.
                \end{align*}
                Since $\overline{a_{N+1}^z} = \langle s_{N+1}^z, h^z_{N+1} \rangle = D^{N+1} s_{N+1}^z(z)$ is a positive real number from the construction for $s^z_{N+1}$,
                \begin{align} \label{33.7}
                    \lim_{z \rightarrow p} a^z_{N+1} = a^p_{N+1}.
                \end{align}
                In conclusion, (\ref{33.4}), (\ref{33.5}), (\ref{33.6}) and (\ref{33.7}) yield
                \begin{align*}
                    \lim_{z \rightarrow p} s^z_{N+1} 
                    = \lim_{z \rightarrow p} \left( \frac{1}{a_{N+1}^z} h^z_{N+1} - \sum_{j=0}^{N} \frac{a_j^z}{a_{N+1}^z} s_j^z  \right)
                    = \lim_{z \rightarrow p} s^p_{N+1}, 
                \end{align*}
                which completes the proof. 
                
                \item[(2)] 
                For each $k \ge 0$ and $l \ge 0$, 
                \begin{align*}
                    |D^l s^z_k(z) - D^l s^p_k(p)| 
                    &=  |D^l s^z_k(z) - D^l s^p_k(z) + D^l s^p_k(z) - D^l s^p_k(p)| \\
                    &\le |D^l s^z_k(z) - D^l s^p_k(z)| + |D^l s^p_k(z) - D^l s^p_k(p)| \\
                    &= |T_z^l (s^z_k - s^p_k)| + |D^l s^p_k(z) - D^l s^p_k(p)| \\
                    &\le \norm{T_z^l}_{*} \norm{s^z_k - s^p_k}_{L^2} + |D^l s^p_k(z) - D^l s^p_k(p)|.
                \end{align*}
                Note that, from Lemma~\ref{lem: T_z continuous}, $\norm{T_z^l}_{*}$ is bounded for all $z \in U_p$, where $U_p$ is a small open neighborhood of $p$.
                Since 
                $$\lim_{z \rightarrow p} \norm{s^z_k - s^p_k}_{L^2}=0 \text{ by } (1), \quad  \lim_{z \rightarrow p} |D^l s^p_k(z) - D^l s^p_k(p)| = 0, $$
                one concludes that 
                $\lim_{z \rightarrow p} |D^l s^z_k(z) - D^l s^p_k(p)| = 0$.
            \end{enumerate}
        \end{proof}

\vspace{5mm}
	
	\section{Bounded domains as statistical models} \label{sec: bounded domains as statistical manifolds}

        \subsection{Bergman statistical models}
        In this subsection, we realize a bounded domain $\O \subset \CC^n$ as a statistical model $(\O, \O, dV, \Phi)$, where $dV$ is the Lebesgue measure on $\CC^n$.
	
	The {\it Poisson-Bergman kernel} $P(z, \xi)$ on $\O$, which is also called the {\it Berezin kernel}, is defined by
	$$ \Poisson(z, \xi) := \frac{|\Bergman(z,\xi)|^2}{\Bergman(z,z)}, $$
	where $\Bergman(z,\xi)$ is the Bergman kernel on $\O$.
	Although the Poisson-Bergman kernel is real-valued, in contrast to the Bergman kernel, it also satisfies the following reproducing property.
	\begin{prop} \label{prop: Poisson reproducing}
		Assume that, for each $z \in \O$, either
		\begin{equation} \label{crucial additional condition}
			\Bergman(\cdot, z) f(\cdot) \in A^2(\O) \quad \text{ or } \quad \overline{\Bergman(\cdot, z) f(\cdot)} \in A^2(\O).
		\end{equation}
		Then
		$$ f(z) = \int_{\O} \Poisson(z,\xi) f(\xi) dV(\xi) $$
		for all $z \in \O$.
	\end{prop}
	\begin{proof}
		For each $z \in \O$,
		\begin{align*}
			\int_{\O} \Poisson(z,\xi) f(\xi) dV(\xi) 
			= \int_{\O} \Bergman(z,\xi) \frac{\Bergman(\xi,z)}{\Bergman(z,z)}f(\xi) dV(\xi) 
			= f(z).
		\end{align*}
		In the second equality, Proposition~\ref{prop: Bergman reproducing} is applied.
	\end{proof}

	Now, we construct the map $\Phi : \O \rightarrow \Prob(\O, dV)$ defined by 
	\begin{equation} \label{def: map into P(D)}
		z \longmapsto \Poisson(z,\xi) dV(\xi) := \frac{|\Bergman(z,\xi)|^2}{\Bergman(z,z)} dV(\xi) .
	\end{equation}
	Note that $\int_{\O} \Poisson(z,\xi)dV(\xi) = 1$ thanks to Proposition~\ref{prop: Poisson reproducing}.
	In order for the map $\Phi$ to be well-defined, we need the following two conditions.
	\begin{enumerate} 
		\item $\Bergman(z,z) > 0$ for all $z \in \O$, 
		\item for fixed $z \in \O$, $P(z, \cdot) > 0$ $dV$-a.e. 
	\end{enumerate}
	When $\O$ is bounded, the condition $(1)$ is guaranteed, 
	and since $A_z = \{  \xi \in \O :  \Bergman(z, \xi) = 0 \}$ is a complex subvariety in $\O$, $(2)$ is also satisfied. \\

    Let $\partial_V$ denote the directional derivative in the direction $V \in T^{1,0}\O$.
    \begin{lem} \label{lem: d_v P continuous}
        For a continuous vector field $V \in T^{1,0}\O$, the map 
        \begin{align*}
            z \longmapsto \partial_V P(z, \cdot) 
        \end{align*}
        is continuous in $L^1$, 
        i.e., $\lim_{z \rightarrow p} \norm{\partial_V P(z, \cdot) - \partial_V P(p, \cdot) }_{L^1(\O, dV)} = 0$ for $p \in \O$.
    \end{lem}
    \begin{proof}
        Let $\{ s^{z}_k \}_{k=0}^{\infty}$ be a special orthonormal basis for $A^2(\O)$ with respect to $z \in \O$ and $V \in T^{1,0}_z\O$. By the direct calculation, we have
        \begin{align} \label{equ: d_v P}
            \partial_V P(z, \xi) 
            &= \frac{\partial_V s^z_1(z)}{s^z_0(z)} s^z_0(\xi) \overline{s^z_1(\xi)}.
        \end{align}
        Then, by letting $A(z) = \partial_V s^z_1(z) / s^z_0(z)$,
        \begin{align*}
            & \| A(z) s^z_0(\cdot) \overline{s^z_1(\cdot)} - A(p) s^p_0(\cdot) \overline{s^p_1(\cdot)} \|_{L^1(\O, dV)} \\
            \le & |A(z) - A(p)| \| s^p_0 \|_{L^2(\O, dV)} \| s^p_1 \|_{L^2(\O, dV)} \\
            + & |A(z)| \| s^z_0 - s^p_0 \|_{L^2(\O, dV)} \| s^p_1 \|_{L^2(\O, dV)} \\
            + & |A(z)| \| s^z_0 \|_{L^2(\O, dV)} \| s^z_1 - s^p_1 \|_{L^2(\O, dV)} .
        \end{align*}
        Therefore, 
        $$\lim_{z \rightarrow p} \norm{\partial_V P(z, \cdot) - \partial_V P(p, \cdot) }_{L^1(\O, dV)} = 0$$
        by Proposition~\ref{prop: s^z continuous}.
    \end{proof}

    \begin{thm} \label{thm: statistical model with regular}
         The quadruple $(\O, \O, dV, \Phi)$ is a statistical model with the regular density function $P(z, \xi)$ given by \eqref{def: map into P(D)}.
    \end{thm}
    \begin{proof}
        For fixed $V \in \CC^n$, by the mean value theorem, there exists $0 < \theta < 1$ depending on $t$ and $\xi$ such that
        $$\frac{P(z+tV, \xi) - P(z, \xi)}{t} = \partial_V P(z+\theta tV, \xi).$$
        Hence, from Lemma~\ref{lem: d_v P continuous},
        \begin{align*}
            &\lim_{t \to 0} \frac{\norm{P(z+tV, \cdot) - P(z, \cdot) - \partial_{tV} P(z, \cdot) }_{L^1(\O, dV)}}{t|V|} \\
            = & \lim_{t \to 0} \frac{\norm{ \partial_V P(z+\theta tV, \cdot) - \partial_V P(z, \cdot) }_{L^1(\O, dV)}}{|V|} = 0.
        \end{align*}
        Therefore, the map $\Phi: \O \rightarrow \Prob(\O)$ is Fr\'echet differentiable and its derivative is given by $(D\Phi)_z(V) = \partial_V P(z, \cdot)dV$. The $C^1$-smoothness of $\Phi$ a direct consequence of Lemma~\ref{lem: d_v P continuous}.

        From \eqref{equ: d_v P}, since $s^z_0, s^z_1 \in L^2(\O, dV)$, we know that $\partial_V P(z, \cdot) \in L^1(\O, dV)$. This implies that the statistical model $(\O, \O, dV, \Phi)$ has the regular density function $P$.
    \end{proof}

    \begin{rmk}
        Note that, for fixed $z \in \O$, the set $\{  \xi \in \O :  \Bergman(z, \xi) = 0 \}$ is a complex subvariety in $\O$ and thus has $dV$-measure zero. 
        Hence, $P(z, \cdot)$ is positive $dV$-almost everywhere. 
        This ensures that the statistical model $(\O, \O, dV, \Phi)$ defined in \eqref{def: map into P(D)} admits a positive regular density function.
    \end{rmk}

    We refer to the statistical model $(\O, \O, dV, \Phi)$ as the \emph{Bergman statistical model}.
 

    \begin{thm}[cf. \cite{BurbeaJacobRaoCRadhakrishna84}] \label{thm: pullback of Fisher is Bergman}
		The Bergman statistical model $(\O, \O, dV, \Phi)$ is $2$-integrable and
		$$  (g_B)_z (V, \overline{V}) = (\Phi^* g_F)_z (V, \overline{V}).
		$$
		for all $z \in \O$ and $V \in T^{1,0}_z\O$. In other words,
            \begin{align*}
                (g_B)_z (V, \overline{V}) 
                = \Cov \left[\partial_{V} \log P(z, \cdot) ,\,  \partial_{V} \log P(z, \cdot) \right].
            \end{align*}
	\end{thm}
	\begin{proof}
        Let $\{ s^{z}_k \}_{k=0}^{\infty}$ be a special orthonormal basis for $A^2(\O)$ with respect to $z \in \O$ and $V \in T^{1,0}_z\O$. 
        Then the direct calculation yields 
        \begin{align*}
            |\partial_{V} \log \Poisson(z,\xi)|^2
            = \frac{|\partial_{V} s^z_1(z)|^2}{|s^z_0(z)|^2} \frac{|s^z_1(\xi)|^2}{|s^z_0(\xi)|^2}
        \end{align*}
        and  
        \begin{align*}
            P(z, \xi) = |s^z_0(\xi)|^2, \quad 
            (g_B)_z(V, \overline{V}) = \frac{|\partial_{V} s^z_1(z)|^2}{|s^z_0(z)|^2}.
        \end{align*}
        Therefore,            
		\begin{align*} 
			(\Phi^* g_F)_z (V, \overline{V})
			&:= \int_{\O} |\partial_{V} \log \Poisson(z,\xi)|^2 \Poisson(z, \xi) dV(\xi)  \\
                &= (g_B)_z(V, \overline{V})  \int_{\O} |s^z_1(\xi)|^2 dV(\xi)  \\
                &= (g_B)_z(V, \overline{V}).  
		\end{align*}
        In the last equality, $ \int_{\O} |s^z_1(\xi)|^2 dV(\xi) = 1$ is used.

        Now,
        \begin{align*}
            \norm{\partial_V \log P(z, \cdot)}^2_{L^2(\O, \Phi(z))} = (g_B)_z(V, \overline{V})
        \end{align*}
        implies that the statistical model $(\O, \O, dV, \Phi)$ is $2$-integrable because the Bergman metric is continuous on $T\O$.
        Moreover, since every $2$-integrable statistical model satisfies $\Exp[\partial_V \log P(z, \cdot)] = 0$ (Remark~\ref{rmk: integral dP = 0}), 
        $$(g_B)_z (V, \overline{V}) = \Cov \left[\partial_{V} \log P(z, \cdot) ,\,  \partial_{V} \log P(z, \cdot) \right].$$
	\end{proof}

        \begin{rmk}
            The proof above is different from the proof in \cite{BurbeaJacobRaoCRadhakrishna84}, which is the following. 
            \begin{align*}
                (\Phi^* g_F)_z (V, \overline{V})
                :&= \int_{\O} (\partial_{V} \log \Poisson(z, \xi)) (\partial_{\overline{V}} \log \Poisson(z, \xi)) \Poisson(z, \xi) dV(\xi) \\
                &= \int_{\O} \partial_{V} \partial_{\overline{V}} \Poisson(z, \xi) dV(\xi) 
                - \int_{\O} (\partial_{V} \partial_{\overline{V}} \log \Poisson(z, \xi)) \Poisson(z, \xi) dV(\xi) \\
                &= \partial_{V} \partial_{\overline{V}} \log \Bergman(z, z) \int_{\O}  \Poisson(z, \xi) dV(\xi) \\
                &= \partial_{V} \partial_{\overline{V}} \log \Bergman(z, z) = (g_B)_z (V, \overline{V}).
            \end{align*}
		In the third equality, $\int_{\O} \partial_{V} \partial_{\overline{V}} \Poisson(z, \xi) dV(\xi) = 0 $ because $\int_{\O}  \Poisson(z, \xi) dV(\xi) = 1$. However, to complete the proof, we need to further show the interchangeability of differentiation and integration, i.e.,
        $$ \partial_{V} \partial_{\overline{V}} \int_{\O} \Poisson(z, \xi) dV(\xi)
        = \int_{\O} \partial_{V} \partial_{\overline{V}} \Poisson(z, \xi) dV(\xi).
        $$
        \end{rmk}
    
	\begin{rmk}
		A Riemannian metric $g$ is Hermitian if and only if $g(\partial_{\alpha}, \partial_{\beta}) = 0$ for all $1 \le \alpha, \beta \le n$, where $\partial_{\alpha} = \frac{\partial}{\partial z_{\alpha}}$ and $\partial_{\beta} = \frac{\partial}{\partial z_{\beta}}$.
		If Condition $C_4$ (Definition~\ref{def: Condition C_k}) is satisfied, then $\Phi^* g_F$ becomes a Hermitian metric by (2) in Lemma \ref{lem: 3-tensors identities}. 
		We do not know whether Condition $C_4$ always hold or not. 
	\end{rmk}

	Next, we investigate the conditions under which the map $\Phi$ is an injective immersion.
    To do this, we first introduce the Bergman-Bochner map.
        Let $\PP( A^2(\O)^*)$ be the projectivization of dual space $A^2(\O)^*$.
	Then there exists a canonical map $\iota_{\O} : \O \rightarrow \PP( A^2(\O)^*) \cong \CC\PP^{\infty} $ defined by $z \mapsto \left[ s_0(z), s_1(z), \cdots \right]$, where $\{s_{j}\}_{j=0}^{\infty}$ is an orthonormal basis for $A^2(\O)$.
	The map $\iota_{\O}$ is called the Bergman-Bochner map (\cite{Kobayashi59}), which can be regarded as a $L^2$-version of the Kodaira map. 
	Since $A^2(\O)$ is a Hilbert space, there exists an isomorphism $\varphi : A^2(\O)^* \rightarrow A^2(\O)$ given by 
	$\sum_{j=0}^{\infty} a_j \inner{ \cdot,  s_j } \mapsto \sum_{j=0}^{\infty} \overline{a}_j s_j$.
	Also, there exists a natural map $\phi$ from $\PP( A^2(\O) )$ to $\Prob(\O,dV)$ defined by 
	$s \mapsto \frac{|s|^2}{\norm{s}^2_{L^2(\O, dV)}} dV$.
	The following proposition implies that the map $\Phi$ is essentially same as the Bergman-Bochner map $\iota_{\O}$.
	
	\begin{prop} \label{prop: essentially the Bergman-Bochner map} 
	    The map $\Phi : \O \rightarrow \Prob(\O,dV)$ defined by (\ref{def: map into P(D)}) is the composition of the following maps:
		$ \Phi : \O \xrightarrow{\iota_{\O}} \PP( A^2(\O)^*) \xrightarrow[\cong]{\varphi} \PP(A^2(\O)) \xrightarrow{\phi} \Prob(\O,dV) $.
	\end{prop}
	\begin{proof}
		For $z \in \O$, 
		\begin{align*}
			(\phi \circ \varphi \circ \iota_{\O})(z) 
			&= (\phi \circ \varphi)\left( \sum_{j=0}^{\infty} s_j(z) \inner{ \cdot,  s_j } \right)
			= \phi \left( \sum_{j=0}^{\infty} \overline{s_j(z)} s_j  \right) \\
			&= \frac{1}{ \sum_{i,j=0}^{\infty} s_i(z) \overline{s_j(z)} \inner{s_j, s_i} } \sum_{i,j=0}^{\infty} \left( s_i(z) \overline{s_j(z)} \overline{s_i(\xi)} s_j(\xi) \right) dV(\xi)  \\
			&= \frac{|\Bergman(z,\xi)|^2}{\Bergman(z,z)} dV(\xi) = \Phi(z).
		\end{align*}
	\end{proof}

	\begin{prop}\label{prop:Kodaira-embedding}
		The following two properties hold.
		\begin{enumerate} 
			\item $\Phi$ is injective if and only if the Bergman-Bochner map $\iota_{\O}$ is injective.
			\item $\Phi$ is an immersion if and only if the Bergman-Bochner map $\iota_{\O}$ is an immersion.
		\end{enumerate}
	\end{prop}
	
	To prove the above proposition, we need the following two lemmas based on Kobayashi (\cite[Theorem 8.1 and 8.2]{Kobayashi59}).
	
	\begin{lem} \label{lem: immersion condition}
		$\iota_{\O}$ is an immersion if and only if
		$\O$ admits the (positive-definite) Bergman metric.
	\end{lem}
	
	\begin{lem}[cf. \cite{Wang15}] \label{lem: injective condition}
		The following are equivalent.
		\begin{enumerate}
			\item $\iota_{\O}$ is injective.
			\item The diastasis function $\Dia(z,\xi) > 0$ for all $z \neq \xi \in \O$.
			\item For any $p \neq q \in \O$, there exists a function $f \in A^2(\O)$ such that $f(p) = 0$ and $f(q) \neq 0$.  
		\end{enumerate}
	\end{lem}
	
	\begin{proof}[Proof of Proposition~\ref{prop:Kodaira-embedding}]
		If $\Phi = \phi \circ \varphi \circ \iota_{\O}$ is an injective immersion, then obviously $\iota_{\O}$ is also an injective immersion. 
		
		Assume that $\iota_{\O}$ is injective and $\Phi(p) = \Phi(q)$ for some $p,q \in \O$. Then
		\begin{align*}
			\frac{|\Bergman(p,\xi)|^2}{\Bergman(p,p)} &= \frac{|\Bergman(q,\xi)|^2}{\Bergman(q,q)} &\text{ for all } \xi \in \O \\
			\Rightarrow \quad \frac{\Bergman(\xi,\xi)}{e^{\Dia(p,\xi)}} &= \frac{\Bergman(\xi,\xi)}{e^{\Dia(q,\xi)}} &\text{ for all } \xi \in \O \\
			\Rightarrow \quad \Dia(p,\xi) &= \Dia(q, \xi) &\text{ for all } \xi \in \O
		\end{align*}
		Plugging $\xi=p$ yields $0 = \Dia(p,p) = \Dia(q,p)$, which implies $p=q$ by the assumption and Lemma~\ref{lem: injective condition}.
		Hence, $\Phi$ is injective.
		
		Suppose that $\iota_{\O}$ is an immersion.
        For the sake of contradiction, assume that $\ker(d\Phi) \neq 0$. Then, for a non-zero vector $V \in \ker(d\Phi)$, 
		\begin{align*}
			g_B(V, \overline{V}) = \Phi^* g_F (V, \overline{V}) = g_F( d\Phi(V), d\Phi(\overline{V})) = 0
		\end{align*}
		by Theorem~\ref{thm: pullback of Fisher is Bergman}. This is a contradiction that the Bergman metric is positive-definite by Lemma~\ref{lem: immersion condition}.
		Hence, $\ker(d\Phi) =0$.
	\end{proof}

    \begin{thm} \label{thm: Phi is an injective immersion}
        For a bounded domain $\O \subset \CC^n$, the map $\Phi$ defined by \eqref{def: map into P(D)} is an injective immersion.
    \end{thm}
    \begin{proof}
        For a bounded domain $\O \subset \mathbb{C}^n$, the Bergman-Bochner map $\iota_{\O}$ is an injective immersion, as the Bergman metric is positive-definite (Lemma~\ref{lem: immersion condition}) and $\O$ satisfies condition $(3)$ in Lemma~\ref{lem: injective condition}. Consequently, our map $\Phi$ is also an injective immersion by Proposition~\ref{prop:Kodaira-embedding}.
    \end{proof}

\vspace{2mm}

    \subsection{Condition ${\bf C_k}$}
    In the next subsection, we introduce a new statistical curvature formula for the Bergman metric. To derive this formula, we primarily utilize the reproducing property (Proposition~\ref{prop: Poisson reproducing}) and the interchangeability of differentiation and integration (\ref{prop: interchangeability of differentiation and integration}).
    To this end, we require a certain technical condition, referred to as Condition $C_k$.
    It is designed to guarantee the finiteness and continuity of all moments of the derivatives of $l = l(z,\xi) = \log P(z, \xi)$ up to total order $k$.

   \begin{defn} \label{def: Condition C_k}
   	A bounded domain $\O \subset \CC^n$ is said to satisfy {\it Condition $C_k$} ($k \in \NN$) if, for each $z \in \O$, 
   	$$\prod_{i=1}^A (\partial^{\Gamma_i} l) \prod_{j=1}^B (\overline{\partial}^{\Lambda_j} l) \in L^1(\O, \Phi(z))$$
   	and the map $\O \rightarrow \RR$ defined by
   	$$ z \mapsto \norm{\prod_{i=1}^A (\partial^{\Gamma_i} l) \prod_{j=1}^B (\overline{\partial}^{\Lambda_j} l)}_{L^1(\O, \Phi(z))} $$
   	is continuous for all multi-indices $\Gamma_1, \cdots \Gamma_A$, $\Lambda_1, \cdots, \Lambda_B \in \NN^n_0$ with $\sum_{i} |\Gamma_i| + \sum_{j} |\Lambda_j| \le k$.
   \end{defn}
   For example, $\O$ is said to satisfy Condition $C_2$ if, for all $1 \le \alpha, \beta \le n$,
   \begin{align*}
   	&\norm{\partial_{\alpha} l}_{L^1(\O, \Phi(z))}, 
   	&\norm{(\partial_{\alpha} l)(\partial_{\beta} l)}_{L^1(\O, \Phi(z))}, \\
   	&\norm{\partial_{\alpha} \partial_{\beta} l}_{L^1(\O, \Phi(z))}, 
   	&\norm{\partial_{\alpha} \partial_{\overline{\beta}} l}_{L^1(\O, \Phi(z))}
   \end{align*}
   are finite and continuous in $z \in \O$. 
   Although, in general, $k$-integrable statistical models (Definition~\ref{def: k-integrable}) do not satisfy Condition $C_k$ for $k \ge 2$, the Bergman statistical models satisfy Condition $C_2$ as follows.

    \begin{prop}
    	A bounded domain $\O \subset \CC^n$ satisfies Condition $C_2$.
    \end{prop}
    \begin{proof}
    	Since the Bergman statistical model is $2$-integrable (Theorem~\ref{thm: pullback of Fisher is Bergman}) and  
    	\begin{align*}
    		\norm{\partial_{\alpha} \partial_{\overline{\beta}} l}_{L^1(\O, \Phi(z))} 
    		= |\partial_{\alpha} \partial_{\overline{\beta}} \log \Bergman(z,z)| \int_{\O} P(z,\xi)dV(\xi) 
    		= |\partial_{\alpha} \partial_{\overline{\beta}} \log \Bergman(z,z)|,
    	\end{align*}
    	it remains to prove that $\norm{\partial_{\alpha} \partial_{\beta} l}_{L^1(\O, \Phi(z))}$ is finite and continuous in $z \in \O$. We prove it when $\alpha=\beta$ for the simplicity.
    	
    	For $z \in \O$, let $\{ s_j^z \}_{j=0}^{\infty}$ be a special orthonormal basis (Section~\ref{subsec: Special orthonormal basis}) for $A^2(\O)$ with respect to $z$ satisfying that
    	$ s_0^z(z) \neq 0$, $s_j^z(z) = 0$ for all $j \ge 1$,
    	$ \partial_{\alpha} s_1^z(z) \neq 0$, $\partial_{\alpha} s_j^z(z) = 0$ for all $j \ge 2$, and
    	$ \partial_{\alpha} \partial_{\alpha} s_2^z(z) \neq 0$, $\partial_{\alpha} \partial_{\alpha} s_j^z(z) = 0$ for all $j \ge 3$.
    	Then
    	\begin{align*}
    		\partial_{\alpha} \partial_{\alpha} l (z, \xi) 
    		&= \left( \frac{\partial_{\alpha} \partial_{\alpha} s_1^z(z)}{s_0^z(z)} - 2\frac{\partial_{\alpha} s_0^z(z) \partial_{\alpha} s_1^z(z)}{s_0^z(z)^2} \right) \frac{\overline{s_1^z(\xi)}}{\overline{s_0^z(\xi)}} + \frac{\partial_{\alpha} \partial_{\alpha} s_2^z(z)}{s_0^z(z)}\frac{\overline{s_2^z(\xi)}}{\overline{s_0^z(\xi)}}
    		- \frac{\partial_{\alpha} s_1^z(z)^2}{s_0^z(z)^2} \frac{\overline{s_1^z(\xi)}^2}{\overline{s_0^z(\xi)}^2} \\
    		&=: A(z) \frac{\overline{s_1^z(\xi)}}{\overline{s_0^z(\xi)}} 
    		+ B(z) \frac{\overline{s_2^z(\xi)}}{\overline{s_0^z(\xi)}}
    		- C(z) \frac{\overline{s_1^z(\xi)}^2}{\overline{s_0^z(\xi)}^2},
    	\end{align*}
    	and, since $P(z,\xi) = |s_0^z(\xi)|^2$,  
    	\begin{align} \label{33.8}
    		&|\partial_{\alpha} \partial_{\alpha} l| P(z,\xi) \\
    		\le& |A(z)|  |s_0^z(\xi)||s_1^z(\xi)|
    		+ |B(z)|  |s_0^z(\xi)||s_2^z(\xi)| 
    		+ |C(z)|  |s_1^z(\xi)|^2 =: F(z, \xi). \nonumber
    	\end{align}
    	Thus, 
    	\begin{align*}
    		&\norm{\partial_{\alpha} \partial_{\alpha} l}_{L^1(\O, \Phi(z))} \\
    		\le& |A(z)| \int_{\O} |s_0^z(\xi)||s_1^z(\xi)|dV(\xi) 
    		+ |B(z)| \int_{\O} |s_0^z(\xi)||s_2^z(\xi)|dV(\xi) 
    		+ |C(z)| \int_{\O} |s_1^z(\xi)|^2dV(\xi) \\
    		\le&  |A(z)| + |B(z)| + |C(z)| 
    		< \infty.
    	\end{align*}
    	
    	To prove the continuity, 
    	\begin{align*}
    		\norm{|A(z)||s_0^z||s_1^z| - |A(p)||s_0^p||s_1^p|}_{L^1(\O, dV)} 
    		\le& |A(z)| \norm{s_0^z}_{L^2(\O,dV)} \norm{s_1^z - s_1^p}_{L^2(\O,dV)} \\
    		+& |A(z)| \norm{s_0^z - s_0^p}_{L^2(\O,dV)} \norm{s_1^p}_{L^2(\O,dV)} \\
    		+& |A(z)-A(p)| \norm{s_0^p}_{L^2(\O,dV)} \norm{s_1^p}_{L^2(\O,dV)} \\ 
    		&\longrightarrow 0
    	\end{align*}
    	as $z \rightarrow p$ by Proposition~\ref{prop: s^z continuous}. 
    	By applying the same argument to $|B(z)||s_0^z(\xi)||s_2^z(\xi)|$ and $|C(z)|  |s_1^z(\xi)|^2$, we have
    	\begin{align*}
    		\lim_{z \rightarrow p} \int_{\O} F(z,\xi) dV(\xi)
    		=  \int_{\O} F(p,\xi) dV(\xi).
    	\end{align*}
    	In conclusion, applying General Lebesgue Dominated Convergence Theorem (\cite[Theorem 19]{royden2010real}) to (\ref{33.8}) yields
    	\[
    	\lim_{z \rightarrow p} \norm{\partial_{\alpha} \partial_{\alpha} l(z, \cdot)}_{L^1(\O, \Phi(z))} 
    	= \norm{\partial_{\alpha} \partial_{\alpha} l(p, \cdot)}_{L^1(\O, \Phi(p))}.
    	\]
    \end{proof}

        It is an interesting question in itself that a bounded domain $\O$ satisfies Condition $C_k$ for all $k \ge 2$.
        When $\O$ is a bounded Hermitian symmetric domain, we have the following positive answer.
        
        \begin{prop} \label{prop: symmetric domain satisfies Condition C_k}
            A bounded Hermitian symmetric domain $\O \subset \CC^n$ satisfies Condition $C_k$ for all $k \in \NN$.
        \end{prop}
        \begin{proof}
            Since the Bergman kernels $\Bergman(z,\xi)$ of irreducible bounded Hermitian symmetric domains $\O$ are explicitly well-known (cf. \cite[Lemma 3.1]{seo2022weakly}), we know that, for fixed $z\in \O$, $\Bergman(z, \cdot)$ is zero-free and continuously extends to the closure $\overline{\O}$ of the domain. Hence, all the integrands are bounded functions on $\overline{\O}$, which implies Condition $C_k$.
            For reducible bounded Hermitian symmetric domains, it follows from the fact that the Bergman kernel $\Bergman_{\O}$ of a product domain $\O = \O_1 \times \O_2$ is $\Bergman_{\O} =  \Bergman_{\O_1} \cdot \Bergman_{\O_2}$. 
        \end{proof}

        The following proposition shows that Condition $C_k$ implies the interchangeability of differentiation and integration.

        \begin{prop} \label{prop: interchangeability of differentiation and integration}
            Suppose that a bounded domain $\O \subset \CC^n$ satisfies Condition $C_{k+1} (k\ge 0)$. Then 
            \begin{align*}
                &\partial_{\alpha}  \int_{\O} \prod_{i=1}^A (\partial^{\Gamma_i} l) \prod_{j=1}^B (\overline{\partial}^{\Lambda_j} l) P(z,\xi)dV(\xi) \\
                = &  \int_{\O} \partial_{\alpha} \left[\prod_{i=1}^A (\partial^{\Gamma_i} l) \prod_{j=1}^B (\overline{\partial}^{\Lambda_j} l) P(z,\xi)  \right] dV(\xi)
            \end{align*}
            for $\alpha = 1, \cdots, n$, where $\Gamma_1, \cdots \Gamma_A$, $\Lambda_1, \cdots, \Lambda_B \in \NN^n_0$ are multi-indices such that $\sum_{i} |\Gamma_i| + \sum_{j} |\Lambda_j| \le k$.
        \end{prop}
\begin{proof}
	Let
	\[
	F(z,\xi) := \prod_{i=1}^A (\partial^{\Gamma_i} l)\prod_{j=1}^B (\overline{\partial}^{\Lambda_j} l)\, P(z,\xi).
	\]
	Consider the following two maps
	\[
    \Psi:\O\to L^1(\O,dV), 
	\qquad
	I: L^1(\O,dV)\to \CC
	\]
	defined by
	\[
    \Psi(z):=F(z,\cdot),
	\qquad
	I(f):=\int_{\O} f(\xi)\,dV(\xi).
	\]
	Note that, from the assumption that \(\O\) satisfies Condition \(C_{k+1}\), we have \(\Psi(z)\in L^1(\O,dV)\).
	Following the same reasoning as in the proof of Theorem~\ref{thm: statistical model with regular} under Condition $C_{k+1}$, the map $\Psi$ is Fr\'echet differentiable with its derivative given by 
    $$(D\Psi)_z(V) = \partial_V F(z, \cdot)$$ 
    for $V \in T_z\O$.
	Furthermore, since \(I\) is a bounded linear operator, it is Fr\'echet differentiable and its derivative at any point is simply the operator itself, i.e.
	\[
	DI = I.
	\]
	Consequently, by the chain rule (\cite[XIII, Section 3]{lang2012real}),
	\begin{align*}
        \partial_{\alpha} (I \circ \Psi)(z) 
        = (DI \circ D\Psi)(\partial_{\alpha})
        = I( \partial_{\alpha} F(z, \cdot) )
        = \int_{\O} \partial_{\alpha} F(z, \xi) dV(\xi),
    \end{align*}
	which completes the proof.
\end{proof}

\vspace{2mm}  
        \subsection{Statistical curvature formula}

        In this subsection, we derive a new statistical curvature formula for the Bergman metric. We frequently apply Proposition~\ref{prop: Poisson reproducing} and Proposition~\ref{prop: interchangeability of differentiation and integration} without explicitly mentioning them.
        The expectation $\Exp[f]$ is defined by 
        \begin{align*}
            \Exp[f] := \int_{\O} f(\xi) P(z,\xi)dV(\xi),
        \end{align*}
        where $P(z, \xi) = \frac{|\Bergman(z,\xi)|^2}{\Bergman(z,z)}$.
        Let $g_B =  g_{\alpha \overline{\beta}} dz_{\alpha} \otimes d\overline{z}_{\beta}$ be the Bergman metric. 
        Recall that $l = l(z,\xi) := \log P(z,\xi)$ and 
        $\partial_{\alpha} := \frac{\partial}{\partial z_{\alpha}}$, $\partial_{\overline{\beta}} := \frac{\partial}{\partial \overline{z}_{\beta}}$.
        
	\begin{lem} \label{lem: 3-tensors identities}
		For $1 \le \alpha, \beta, \gamma \le n$, under additional assumptions, the following identities hold.
		\begin{enumerate}
			\item $\Exp[(\partial_{\alpha} l)(\partial_{\overline{\beta}} l)] = 
			- \Exp[(\partial_{\alpha} \partial_{\overline{\beta}} l)] = g_{\alpha\overline{\beta}}$.
			\item Under Condition $C_4$, $\Exp[(\partial_{\alpha} l)(\partial_{\beta} l)] = \Exp[(\partial_{\overline{\alpha}} l)(\partial_{\overline{\beta}} l)] =  0$.
			\item Under Condition $C_4$, $\Exp[(\partial_{\alpha}\partial_{\beta} l)] = \Exp[(\partial_{\overline{\alpha}}\partial_{\overline{\beta}} l)] = 0$.
			\item Under Condition {$C_6$}, $\Exp[(\partial_{\alpha} l)(\partial_{\beta} l)(\partial_{\gamma} l)] = \Exp[(\partial_{\overline{\alpha}} l)(\partial_{\overline{\beta}} l)(\partial_{\overline{\gamma}} l)] =  0$.
			\item Under Condition {$C_6$}, $\Exp[(\partial_{\alpha}\partial_{\beta} l)(\partial_{\gamma} l)] = \Exp[(\partial_{\overline{\alpha}} \partial_{\overline{\beta}} l)(\partial_{\overline{\gamma}} l)] = 0$.
			\item $\Exp[(\partial_{\alpha}\partial_{\overline{\beta}} l)(\partial_{\gamma} l)] = \Exp[(\partial_{\alpha}\partial_{\overline{\beta}} l)(\partial_{\overline{\gamma}} l)] = 0$.
			\item Under Condition $C_4$, $\Exp[(\partial_{\alpha} l)(\partial_{\beta} l)(\partial_{\overline{\gamma}} l)] = \Exp[(\partial_{\overline{\alpha}} l)(\partial_{\overline{\beta}} l)(\partial_{\gamma} l)] = 0$.
			\item Under Condition $C_4$, $\Exp[(\partial_{\alpha}\partial_{\beta} l)(\partial_{\overline{\gamma}} l)] = \partial_{\beta} g_{\alpha \overline{\gamma}}$.
			\item Under Condition $C_6$, $\Exp[(\partial_{\alpha}\partial_{\beta}\partial_{\gamma} l)] 
			= \Exp[(\partial_{\overline{\alpha}}\partial_{\overline{\beta}}\partial_{\overline{\gamma}} l)] = 0$.
			\item Under Condition $C_4$, $\Exp[(\partial_{\alpha}\partial_{\beta}\partial_{\overline{\gamma}} l)] 
			= - \partial_{\beta} g_{\alpha \overline{\gamma}}$.
		\end{enumerate}
	\end{lem}
	\begin{proof}
		\begin{enumerate}
			\item 
                It follows from Theorem \ref{thm: pullback of Fisher is Bergman}, and              $\Exp[(\partial_{\alpha} \partial_{\overline{\beta}} l)] = -g_{\alpha\overline{\beta}}\Exp[1] = -g_{\alpha\overline{\beta}}$.
			\item 
                Since $\partial_{\alpha} l(z,\xi)\partial_{\beta} l(z,\xi)\Bergman(z,\xi)$ is anti-holomorphic in $\xi$, by applying Proposition~\ref{prop: Poisson reproducing},
			\begin{align*}
				&\Exp[(\partial_{\alpha} l)(\partial_{\beta} l)] \\
				=& \int_{\O}  (\partial_{\alpha} \log\Bergman(z,\xi) - \partial_{\alpha} \log\Bergman(z,z)) (\partial_{\beta} \log\Bergman(z,\xi) - \partial_{\beta} \log\Bergman(z,z)) P(z,\xi) dV(\xi) \\
				=& (\partial_{\alpha} \log\Bergman(z,z) - \partial_{\alpha} \log\Bergman(z,z)) (\partial_{\beta} \log\Bergman(z,z) - \partial_{\beta} \log\Bergman(z,z)) = 0.
			\end{align*}
			\item 
                Since $\Exp[\partial_{\alpha} l] = 0$ (Remark~\ref{rmk: integral dP = 0}), 
                \begin{align*}
                    0 = \partial_{\beta} \Exp[\partial_{\alpha} l] 
                    = \Exp[\partial_{\alpha} \partial_{\beta} l] +  \Exp[(\partial_{\alpha} l)(\partial_{\beta} l)].
                \end{align*}
                Hence, $\Exp[\partial_{\alpha} \partial_{\beta} l] = 0$ by $(2)$.
			\item
                It follows from the same argument as (2).
                \item 
                It follows from the same argument as (2).
                \item 
                $\Exp[(\partial_{\alpha}\partial_{\overline{\beta}} l)(\partial_{\gamma} l)] 
                = -g_{\alpha \overline{\beta}} \Exp[\partial_{\gamma} l] = 0 .$
			\item 
                From $(2)$ and $(6)$,
			\begin{align*}
				0 &= \partial_{\overline{\gamma}} \Exp[(\partial_{\alpha} l)(\partial_{\beta} l)]  \\
				&= \Exp[(\partial_{\alpha}\partial_{\overline{\gamma}} l)(\partial_{\beta} l)] + \Exp[(\partial_{\alpha} l)(\partial_{\beta}\partial_{\overline{\gamma}} l)] + \Exp[(\partial_{\alpha} l)(\partial_{\beta} l)(\partial_{\overline{\gamma}} l)] \\
				&= \Exp[(\partial_{\alpha} l)(\partial_{\beta} l)(\partial_{\overline{\gamma}} l)].
			\end{align*}
			\item 
                From $(1)$, $(6)$ and $(7)$, 
			\begin{align*}
				\partial_{\beta} g_{\alpha \overline{\gamma}} &= \partial_{\beta} \Exp[(\partial_{\alpha} l)(\partial_{\overline{\gamma}} l)] \\
				&= \Exp[(\partial_{\alpha}\partial_{\beta} l)(\partial_{\overline{\gamma}} l)] + \Exp[(\partial_{\beta}\partial_{\overline{\gamma}} l)(\partial_{\alpha} l)] 
				+ \Exp[(\partial_{\alpha} l)(\partial_{\beta} l)(\partial_{\overline{\gamma}} l)] \\
				&= \Exp[(\partial_{\alpha}\partial_{\beta} l)(\partial_{\overline{\gamma}} l)].
			\end{align*}
			\item 
                From $(3)$ and $(5)$,
			\begin{align*}
				0 = \partial_{\gamma} \Exp[(\partial_{\alpha}\partial_{\beta} l)] 
				= \Exp[(\partial_{\alpha}\partial_{\beta}\partial_{\gamma} l)] 
				+ \Exp[(\partial_{\alpha}\partial_{\beta} l)(\partial_{\gamma} l)]
				= \Exp[(\partial_{\alpha}\partial_{\beta}\partial_{\gamma} l)] .
			\end{align*}
			\item 
                From $(3)$ and $(8)$, 
			\begin{align*}
				0 = \partial_{\overline{\gamma}} \Exp[(\partial_{\alpha}\partial_{\beta} l)] 
				= \Exp[(\partial_{\alpha}\partial_{\beta}\partial_{\overline{\gamma}} l)] 
				+ \Exp[(\partial_{\alpha}\partial_{\beta} l)(\partial_{\overline{\gamma}} l)]
				= \Exp[(\partial_{\alpha}\partial_{\beta}\partial_{\overline{\gamma}} l)]
				+ \partial_{\beta} g_{\alpha \overline{\gamma}}.
			\end{align*}
		\end{enumerate}
	\end{proof}

	\begin{lem} \label{lem: 4-tensors identities}
		For $1 \le \alpha, \beta, \gamma, \delta \le n$, under additional assumptions, the following identities hold.
		\begin{enumerate}
			\item 
                $\Exp[(\partial_{\alpha}\partial_{\beta} \partial_{\overline{\gamma}} l)(\partial_{\overline{\delta}} l)] = \Exp[(\partial_{\alpha}\partial_{\overline{\beta}} \partial_{\overline{\gamma}} l)(\partial_{\delta} l)] = 0,   $
			\item 
                $\Exp[(\partial_{\alpha}\partial_{\overline{\beta}} l)(\partial_{\gamma} l)(\partial_{\overline{\delta}} l)]  = - g_{\alpha\overline{\beta}} g_{\gamma\overline{\delta}},     $
			\item 
                $\Exp[(\partial_{\alpha}\partial_{\overline{\beta}} l)(\partial_{\gamma}\partial_{\overline{\delta}} l)] = g_{\alpha\overline{\beta}} g_{\gamma\overline{\delta}},   $
			\item 
                Under Condition $C_4$, \\
                $ \Exp[(\partial_{\alpha} l)(\partial_{\beta} l)(\partial_{\overline{\gamma}} \partial_{\overline{\delta}} l)] + \Exp[(\partial_{\alpha} l)(\partial_{\beta} l)(\partial_{\overline{\gamma}} l)(\partial_{\overline{\delta}} l)] 
			=  g_{\alpha \overline{\delta}} g_{\beta \overline{\gamma}} + g_{\beta \overline{\delta}} g_{\alpha \overline{\gamma}},  $
			\item 
                Under Condition $C_4$, \\
                $ \partial_{\beta} \partial_{\overline{\delta}} g_{\alpha \overline{\gamma}} 
			= \Exp[(\partial_{\alpha}\partial_{\beta} l)(\partial_{\overline{\gamma}}\partial_{\overline{\delta}} l)] 
			+ \Exp[(\partial_{\alpha}\partial_{\beta} l)(\partial_{\overline{\gamma}} l)(\partial_{\overline{\delta}} l)]. $	
		\end{enumerate}
	\end{lem}
	\begin{proof}
		\begin{enumerate}
			\item 
                $\Exp[(\partial_{\alpha}\partial_{\beta} \partial_{\overline{\gamma}} l)(\partial_{\overline{\delta}} l)] 
			= -\partial_{\alpha} g_{\beta \overline{\gamma}} \Exp[(\partial_{\overline{\delta}} l)] = 0.$
			\item 
                $\Exp[(\partial_{\alpha}\partial_{\overline{\beta}} l)(\partial_{\gamma} l)(\partial_{\overline{\delta}} l)]  
			= g_{\alpha \overline{\beta}} \Exp[(\partial_{\gamma} l)(\partial_{\overline{\delta}} l)] = - g_{\alpha\overline{\beta}} g_{\gamma\overline{\delta}}.  $
			\item 
                $\Exp[(\partial_{\alpha}\partial_{\overline{\beta}} l)(\partial_{\gamma}\partial_{\overline{\delta}} l)] 
			= (-g_{\alpha\overline{\beta}}) (-g_{\gamma\overline{\delta}}) \Exp[1] = g_{\alpha\overline{\beta}} g_{\gamma\overline{\delta}}.  $
			\item 
                From $(7)$ in Lemma~\ref{lem: 3-tensors identities} and $(2)$,
			\begin{align*}
				0 &= \partial_{\overline{\delta}} \Exp[(\partial_{\alpha} l)(\partial_{\beta} l)(\partial_{\overline{\gamma}} l)] \\
				&= \Exp[(\partial_{\alpha} \partial_{\overline{\delta}} l)(\partial_{\beta} l)(\partial_{\overline{\gamma}} l)]  
				+  \Exp[(\partial_{\alpha} l)(\partial_{\beta} \partial_{\overline{\delta}} l)(\partial_{\overline{\gamma}} l)]  \\
				&+ \Exp[(\partial_{\alpha} l)(\partial_{\beta} l)(\partial_{\overline{\gamma}} \partial_{\overline{\delta}} l)] 
				+ \Exp[(\partial_{\alpha} l)(\partial_{\beta} l)(\partial_{\overline{\gamma}} l)(\partial_{\overline{\delta}} l)]  \\
				&= - g_{\alpha \overline{\delta}} g_{\beta \overline{\gamma}} - g_{\beta \overline{\delta}} g_{\alpha \overline{\gamma}}
				+ \Exp[(\partial_{\alpha} l)(\partial_{\beta} l)(\partial_{\overline{\gamma}} \partial_{\overline{\delta}} l)] 
				+ \Exp[(\partial_{\alpha} l)(\partial_{\beta} l)(\partial_{\overline{\gamma}} l)(\partial_{\overline{\delta}} l)]. 
			\end{align*}
			\item 
                From $(8)$ in Lemma~\ref{lem: 3-tensors identities} and $(1)$,
			\begin{align*}
				\partial_{\beta} \partial_{\overline{\delta}} g_{\alpha \overline{\gamma}} 
				&= \partial_{\overline{\delta}} \Exp[(\partial_{\alpha}\partial_{\beta} l)(\partial_{\overline{\gamma}} l)] \\
				&= \Exp[(\partial_{\alpha}\partial_{\beta}\partial_{\overline{\delta}} l)(\partial_{\overline{\gamma}} l)] 
				+ \Exp[(\partial_{\alpha}\partial_{\beta} l)(\partial_{\overline{\gamma}}\partial_{\overline{\delta}} l)] 
				+ \Exp[(\partial_{\alpha}\partial_{\beta} l)(\partial_{\overline{\gamma}} l)(\partial_{\overline{\delta}} l)] \\
				&= \Exp[(\partial_{\alpha}\partial_{\beta} l)(\partial_{\overline{\gamma}}\partial_{\overline{\delta}} l)] 
				+ \Exp[(\partial_{\alpha}\partial_{\beta} l)(\partial_{\overline{\gamma}} l)(\partial_{\overline{\delta}} l)].
			\end{align*}
		\end{enumerate}
	\end{proof}


        The curvature 4-tensor of the Bergman metric $g_B =  g_{\alpha \overline{\beta}} dz_{\alpha} \otimes d\overline{z}_{\beta}$ is defined by
        \begin{equation*}
            R_{\alpha \overline{\beta} \gamma \overline{\delta}}
            = - \partial_{\alpha} \partial_{\overline{\beta}} g_{\gamma \overline{\delta}}
            +  g^{\tau \overline{\epsilon}} 
            (\partial_{\alpha} g_{\gamma \overline{\epsilon}})
            (\partial_{\overline{\beta}} g_{\tau \overline{\delta}}),
        \end{equation*}
        where $g^{\alpha \overline{\epsilon}} g_{\beta \overline{\epsilon}} = \delta^{\alpha}_{\beta}$. 
        Here the Einstein summation convention is applied. 
        The covariance is defined by
        \begin{align*}
            \Cov[f, g] 
            := \Exp\left[(f - \Exp[f]) \overline{(g - \Exp[g])} \right]
            = \Exp[f \overline{g}] - \Exp[f] \overline{\Exp[g]},
        \end{align*}
        where $\Exp[f] = \int f(\xi) P(z,\xi)dV(\xi)$.

\begin{thm} \label{thm: statistical curvature formula}
	Suppose that a bounded domain $\O \subset \CC^n$ satisfies Condition $C_4$. Then
	the curvature 4-tensor $R_{\alpha \overline{\beta} \gamma \overline{\delta}}$ of the Bergman metric $g_B =  g_{\alpha \overline{\beta}} dz_{\alpha} \otimes d\overline{z}_{\beta}$ is expressed by
	\begin{align} \label{equ: curvature formula}
		R_{\alpha \overline{\beta} \gamma \overline{\delta}}
		&= g_{\alpha \overline{\beta}} g_{\gamma \overline{\delta}} 
		+ g_{\alpha \overline{\delta}} g_{\gamma \overline{\beta}}   \\
		&- \Cov\left[ \partial_{\alpha} \partial_{\gamma} l + (\partial_{\alpha} l)(\partial_{\gamma} l) -  \Gamma_{\alpha \gamma}^i (\partial_i l), \,
		\partial_{\beta} \partial_{\delta} l + (\partial_{\beta} l)(\partial_{\delta} l) - \Gamma_{\beta \delta}^j (\partial_j l) \right], \nonumber 
	\end{align}
	where $\Gamma^{i}_{\alpha \gamma} = g^{i \overline{\epsilon}}(\partial_{\alpha} g_{\gamma \overline{\epsilon}})$ is the Christoffel symbol for $g_B$.
	Moreover,
	$$N_{\alpha \gamma} := \partial_{\alpha} \partial_{\gamma} l + (\partial_{\alpha} l)(\partial_{\gamma} l) - \Gamma_{\alpha \gamma}^i (\partial_i l) \in T_{\Phi(z)} \Prob(\O)$$
	and 
	$$ N_{\alpha \gamma} \perp T_{\Phi(z)} \Phi(\O).
	$$
\end{thm}

\begin{proof}
	
	We first calculate the covariant part of the right-hand side of the equation (\ref{equ: curvature formula}):
	\begin{align*}
		\Cov\left[ \partial_{\alpha} \partial_{\gamma} l + (\partial_{\alpha} l)(\partial_{\gamma} l) - g^{i \overline{\epsilon}}(\partial_{\alpha} g_{\gamma \overline{\epsilon}}) (\partial_i l), \,
		\partial_{\beta} \partial_{\delta} l + (\partial_{\beta} l)(\partial_{\delta} l) - g^{j \overline{\tau}}(\partial_{\beta} g_{\delta \overline{\tau}}) (\partial_j l) \right]. \\
	\end{align*}
	Together with $(2)$ and $(3)$ in Lemma~\ref{lem: 3-tensors identities},
	\begin{align*}
		&1.\, 
		\Cov\left[ \partial_{\alpha} \partial_{\gamma} l, \partial_{\beta} \partial_{\delta} l \right] 
		= \Exp[ (\partial_{\alpha} \partial_{\gamma} l) (\partial_{\overline{\beta}} \partial_{\overline{\delta}} l) ],  \\
		&2.\,
		\Cov\left[ \partial_{\alpha} \partial_{\gamma} l, (\partial_{\beta} l)(\partial_{\delta} l) \right] 
		= \Exp[ (\partial_{\alpha} \partial_{\gamma} l) (\partial_{\overline{\beta}} l)(\partial_{\overline{\delta}} l) ], \\ 
		&3.\,
		\Cov\left[ \partial_{\alpha} \partial_{\gamma} l, - g^{j \overline{\tau}}(\partial_{\beta} g_{\delta \overline{\tau}}) \partial_j l \right]
		= - g^{\overline{j} \tau}(\partial_{\overline{\beta}} g_{\overline{\delta} \tau}) \Exp[ (\partial_{\alpha} \partial_{\gamma} l) (\partial_{\overline{j}} l) ] 
		= - g^{\overline{j} \tau}(\partial_{\overline{\beta}} g_{\tau \overline{\delta}} ) (\partial_{\alpha} g_{\gamma \overline{j}}) \\
		& \hspace{122mm} (8) \text{ in Lemma}~\ref{lem: 3-tensors identities}, \\
		&4.\, 
		\Cov\left[ (\partial_{\alpha} l)(\partial_{\gamma} l), \partial_{\beta} \partial_{\delta} l \right]
		= \Exp[ (\partial_{\alpha} l)(\partial_{\gamma} l) (\partial_{\overline{\beta}} \partial_{\overline{\delta}} l) ], \\ 
		&5.\,
		\Cov\left[ (\partial_{\alpha} l)(\partial_{\gamma} l), (\partial_{\beta} l)(\partial_{\delta} l) \right]
		= \Exp[ (\partial_{\alpha} l)(\partial_{\gamma} l) (\partial_{\overline{\beta}} l)(\partial_{\overline{\delta}} l) ], \\
		&6.\,
		\Cov\left[ (\partial_{\alpha} l)(\partial_{\gamma} l), - g^{j \overline{\tau}}(\partial_{\beta} g_{\delta \overline{\tau}}) \partial_j l  \right] 
		= - g^{\overline{j} \tau}(\partial_{\overline{\beta}} g_{\overline{\delta} \tau}) E [(\partial_{\alpha} l)(\partial_{\gamma} l) (\partial_{\overline{j}} l) ] = 0 \\
		& \hspace{122mm} (7) \text{ in Lemma}~\ref{lem: 3-tensors identities}, \\
		&7.\,
		\Cov\left[ - g^{i \overline{\epsilon}}(\partial_{\alpha} g_{\gamma \overline{\epsilon}}) \partial_i l, \partial_{\beta} \partial_{\delta} l \right]
		= - g^{i \overline{\epsilon}}(\partial_{\alpha} g_{\gamma \overline{\epsilon}}) \Exp[ (\partial_i l) (\partial_{\overline{\beta}} \partial_{\overline{\delta}} l) ] 
		= - g^{i \overline{\epsilon}}(\partial_{\alpha} g_{\gamma \overline{\epsilon}}) (\partial_{\overline{\beta}} g_{i \overline{\delta}}) \\
		& \hspace{122mm} (8) \text{ in Lemma}~\ref{lem: 3-tensors identities},  \\
		&8.\,
		\Cov\left[ - g^{i \overline{\epsilon}}(\partial_{\alpha} g_{\gamma \overline{\epsilon}}) \partial_i l, (\partial_{\beta} l)(\partial_{\delta} l) \right]
		= - g^{i \overline{\epsilon}}(\partial_{\alpha} g_{\gamma \overline{\epsilon}}) \Exp[ (\partial_i l) (\partial_{\overline{\beta}} l)(\partial_{\overline{\delta}} l) ] = 0 \\
		& \hspace{122mm} (7) \text{ in Lemma}~\ref{lem: 3-tensors identities},  \\
		&9.\,
		\Cov\left[ - g^{i \overline{\epsilon}}(\partial_{\alpha} g_{\gamma \overline{\epsilon}}) \partial_i l, - g^{j \overline{\tau}}(\partial_{\beta} g_{\delta \overline{\tau}}) \partial_j l \right] \\
		&\hspace{3mm} = g^{i \overline{\epsilon}}(\partial_{\alpha} g_{\gamma \overline{\epsilon}}) g^{\overline{j} \tau}(\partial_{\overline{\beta}} g_{\overline{\delta} \tau}) \Exp[ (\partial_i l)  (\partial_{\overline{j}} l) ] 
		= g^{i \overline{\epsilon}}(\partial_{\alpha} g_{\gamma \overline{\epsilon}}) 
		(\partial_{\overline{\beta}} g_{i \overline{\delta}} ) \\
		& \hspace{122mm} (1) \text{ in Lemma}~\ref{lem: 3-tensors identities}.  
	\end{align*}
	In conclusion, from $(4)$ and $(5)$ in Lemma~\ref{lem: 4-tensors identities},
	\begin{align*}
		&\Cov\left[ \partial_{\alpha} \partial_{\gamma} l + (\partial_{\alpha} l)(\partial_{\gamma} l) - g^{i \overline{\epsilon}}(\partial_{\alpha} g_{\gamma \overline{\epsilon}}) (\partial_i l), \,
		\partial_{\beta} \partial_{\delta} l + (\partial_{\beta} l)(\partial_{\delta} l) - g^{j \overline{\tau}}(\partial_{\beta} g_{\delta \overline{\tau}}) (\partial_j l) \right] \\
		=& \Exp[ (\partial_{\alpha} \partial_{\gamma} l) (\partial_{\overline{\beta} \overline{\delta}} l) ]  
		+ \Exp[ (\partial_{\alpha} \partial_{\gamma} l) (\partial_{\overline{\beta}} l)(\partial_{\overline{\delta}} l) ] 
		- g^{\overline{j} \tau}(\partial_{\overline{\beta}} g_{\tau \overline{\delta}} ) (\partial_{\alpha} g_{\gamma \overline{j}})  \\
		+& \Exp[ (\partial_{\alpha} l)(\partial_{\gamma} l) (\partial_{\overline{\beta}} \partial_{\overline{\delta}} l) ] 
		+ \Exp[ (\partial_{\alpha} l)(\partial_{\gamma} l) (\partial_{\overline{\beta}} l)(\partial_{\overline{\delta}} l) ] 
		- g^{i \overline{\epsilon}}(\partial_{\alpha} g_{\gamma \overline{\epsilon}}) (\partial_{\overline{\beta}} g_{i \overline{\delta}})
		+ g^{i \overline{\epsilon}}(\partial_{\alpha} g_{\gamma \overline{\epsilon}}) 
		(\partial_{\overline{\beta}} g_{i \overline{\delta}} )  \\
		=& \partial_{\gamma} \partial_{\overline{\delta}} g_{\alpha \overline{\beta}} - g^{\overline{j} \tau}(\partial_{\overline{\beta}} g_{\tau \overline{\delta}} ) (\partial_{\alpha} g_{\gamma \overline{j}})  
		+ g_{\alpha \overline{\beta}} g_{\gamma \overline{\delta} } + g_{\alpha \overline{\delta} } g_{\gamma \overline{\beta} } . \\
	\end{align*}
	
	Now, since
	\begin{align*}
		\Exp[N_{\alpha \gamma}] = \Exp[\partial_{\alpha} \partial_{\gamma} l] + \Exp[(\partial_{\alpha} l)(\partial_{\gamma} l)] - \Gamma_{\alpha \gamma}^i \Exp[\partial_i l] = 0
	\end{align*}
	and
	\begin{align*}
		g_F(N_{\alpha \gamma}, \partial_{\beta} l) 
		&= \Exp[(\partial_{\alpha} \partial_{\gamma} l)(\partial_{\overline{\beta}} l)] + \Exp[(\partial_{\alpha} l)(\partial_{\gamma} l)(\partial_{\overline{\beta}} l)] - \Gamma_{\alpha \gamma}^i \Exp[(\partial_i l)(\partial_{\overline{\beta}} l)] \\
		&= \partial_{\alpha} g_{\gamma \overline{\beta}}   
		- g^{i \overline{\epsilon}}(\partial_{\alpha} g_{\gamma \overline{\epsilon}}) g_{i \overline{\beta}} = 0
	\end{align*}
	for all $1 \le \alpha, \beta, \gamma \le n$, one concludes that 
    \begin{align*}
        N_{\alpha \gamma} \in T_{\Phi(z)} \Prob(\O) \quad \text{ and } \quad N_{\alpha \gamma} \perp T_{\Phi(z)} \Phi(\O).
    \end{align*}
\end{proof}

\begin{cor} \label{cor: holomorphic sectional curvature statistical formula}
	Suppose that a bounded domain $\O \subset \CC^n$ satisfies Condition $C_4$. Then
    the holomorphic sectional curvature $H(\partial_{\alpha}) := R_{\alpha \overline{\alpha} \alpha \overline{\alpha}} / (g_{\alpha \overline{\alpha}})^2$ of the Bergman metric $g_B =  g_{\alpha \overline{\beta}} dz_{\alpha} \otimes d\overline{z}_{\beta}$ is expressed by
        \begin{align*} 
            H(\partial_{\alpha})
            = 2 
            - \frac{\Var\left[ \partial_{\alpha} \partial_{\alpha} l + (\partial_{\alpha} l)(\partial_{\alpha} l) - \Gamma_{\alpha \alpha}^i (\partial_i l) \right]}{(g_{\alpha \overline{\alpha}})^2}.
        \end{align*}
\end{cor}

\begin{cor}[cf. \cite{Kobayashi59}] \label{cor: holomorphic sectional curvature upper bound}
	Suppose that a bounded domain $\O \subset \CC^n$ satisfies Condition $C_4$. Then the holomorphic sectional curvature of the Bergman metric is always less or equal to $2$.
\end{cor}

\vspace{2mm}  
        \subsection{Information--geometric structures}

        In this subsection, we calculate various important information--geometric structures for the Bergman statistical models $\Phi: \O \rightarrow \Prob(\O,dV)$, 
        such as the Amari-Chentsov tensor $T$ (Definition \ref{def: Amari-Chentsov tensor}), 
	the $\alpha$-connection $\nabla^{(\alpha)}$ (Definition \ref{def: alpha connection}), the Kullback–Leibler (KL) divergence $D^{(\pm 1)}$ (Example \ref{ex: alpha divergence}) and the Shannon entropy $H$.

\begin{thm}\label{thm: pullback of Information--geometric concepts} 
	The following properties hold.
	\begin{enumerate}
		\item 
		Under Condition {$C_6$}, 
        $$\Phi^* T = 0.$$
		\item 
		Under Condition {$C_6$}, 
        $$\Phi^* \nabla^{(\alpha)} = \nabla^{LC}$$
        for all $\alpha \in \RR$, where $\nabla^{LC}$ is the Levi-Civita connection of the Bergman metric.
		\item 
        Under the condition that $\Bergman(\cdot, z) \log  \frac{\Bergman(\cdot,z)}{\Bergman(\cdot,w)} \in A^2(\O)$,
		$$\Phi^* D^{(1)}(z,w) = \Dia(z,w) = \Phi^* D^{(-1)}(z,w)$$ 
        for all $z,w \in \Omega$, where $\Dia$ is the diastasis function (Section~\ref{sec: Bergman geometry}).
		\item 
        Under the condition that $\Bergman(\cdot, z) \log \Bergman(\cdot,z) \in A^2(\O)$,
        $$H(\Phi(z)) = - \log \Bergman(z,z)$$
        for all $z \in \Omega$. 
	\end{enumerate}
\end{thm}

\begin{proof} 
	\begin{enumerate}
		\item 
		Since
		\[
		\overline{\Exp[(\partial_{\alpha} l)(\partial_{\beta} l)(\partial_{\gamma} l)]}
		=
		\Exp[(\partial_{\overline{\alpha}} l)(\partial_{\overline{\beta}} l)(\partial_{\overline{\gamma}} l)],
		\]
		and the Amari-Chentsov tensor is symmetric, it is enough to show that, for $1 \le \alpha, \beta, \gamma \le n$, 
		\begin{align*}
			\Exp[(\partial_{\alpha} l)(\partial_{\beta} l)(\partial_{\gamma} l)] = 0
			\quad \text{and} \quad
			\Exp[(\partial_{\alpha} l)(\partial_{\beta} l)(\partial_{\overline{\gamma}} l)] = 0,
		\end{align*}
		which follow from (4) and (7) in Lemma~\ref{lem: 3-tensors identities}, respectively.
		
		\item
		It follows from (1) that the Amari-Chentsov tensor $T$ vanishes on $\Phi(\O)$.
		
		\item 
		For $z, w \in \O$,  
		\begin{align*}
			(\Phi^*D^{(-1)})(z, w) 
			:=& \int_{\O} P(z,\xi) \log \frac{P(z,\xi)}{P(w,\xi)} \, dV(\xi) \\
			=& \int_{\O} P(z,\xi) \log \frac{\Bergman(w,w)}{\Bergman(z,z)} \, dV(\xi) 
			+ \int_{\O} P(z,\xi) \log \frac{\Bergman(\xi,z)}{\Bergman(\xi,w)} \, dV(\xi) \\
			&\quad
			+ \int_{\O} P(z,\xi) \log \frac{\Bergman(z,\xi)}{\Bergman(w,\xi)} \, dV(\xi) \\
			=& \log \frac{\Bergman(w,w)}{\Bergman(z,z)}
			+ \log \frac{\Bergman(z,z)}{\Bergman(z,w)}
			+ \log \frac{\Bergman(z,z)}{\Bergman(w,z)} \\
			=& \log \frac{\Bergman(z,z)\Bergman(w,w)}{|\Bergman(z,w)|^2}
			= \Dia(z,w).
		\end{align*}
		In the third equality, we apply Proposition~\ref{prop: Poisson reproducing}. Finally,
		\[
		\Phi^* D^{(1)}(z,w)
		=
		\Phi^* D^{(-1)}(w,z)
		=
		\Dia(w,z)
		=
		\Dia(z,w),
		\]
		because the diastasis function is symmetric.
		
		\item 
		For $z \in \O$,
            \begin{align*}
			H(\Phi(z))  
			:=& -\int_{\O} P(z,\xi) \log P(z,\xi) dV(\xi) \\
			=& -\int_{\O} P(z,\xi) \log \Bergman(z,\xi) dV(\xi) 
			- \int_{\O} P(z,\xi) \log \Bergman(\xi,z)  dV(\xi) \\
                &+ \int_{\O} P(z,\xi) \log \Bergman(z,z)  dV(\xi) \\
			=& -\log \Bergman(z,z) - \log \Bergman(z,z) + \log \Bergman(z,z) \\
			=& -\log \Bergman(z,z).
		\end{align*}
	\end{enumerate}
\end{proof}

\begin{cor} \label{cor: not exponential family}
	Under Condition $C_6$, 
	$\Phi(\O)$ is never an exponential family in $\Prob(\O,dV)$.
	\begin{proof}
		If $\Phi(\Omega)$ is an exponential family, then the $1$-connection $\Phi^* \nabla^{(1)}$ must be flat. By (2) in Theorem~\ref{thm: pullback of Information--geometric concepts}, this means that the Levi-Civita connection for the Bergman metric must be flat. This is a contradiction (\cite[Theorem 7.1]{huang2023bergman}).
	\end{proof}
\end{cor}

	Although we could not calculate the pullback of $\alpha$-divergence for all $\alpha \in \RR$ on $\Phi(\O)$ except for the cases $\alpha=1, -1$,
	one shows that the $\alpha$-divergence is invariant under all biholomorphic functions of $\O$ for all $\alpha \in \RR$.

\begin{prop}\label{prop:alpha-divergence}
	For all $\alpha \in \RR$, the pullback of the $\alpha$-divergence is biholomorphic invariant, i.e., for a biholomorphism $f : \O \rightarrow \O$, $\Phi^* D^{(\alpha)}(f(z),f(w)) = \Phi^* D^{(\alpha)}(z,w)$.
\end{prop}
\begin{proof}
	\begin{align*}
		\Phi^* D^{(\alpha)}(f(z),f(w))  
		&= \int_{\O}  \chi^{(\alpha)} \left(\frac{P(f(w),\xi)}{P(f(z),\xi)}\right) P(f(z),\xi) dV(\xi) \\
		&= \int_{\O}  \chi^{(\alpha)} \left( \frac{e^{\Dia(f(w),\xi)}}{e^{\Dia(f(z),\xi)}} \right) e^{\Dia(f(z),\xi)} \Bergman(\xi,\xi) dV(\xi) \\
		&= \int_{\O}  \chi^{(\alpha)} \left( \frac{e^{\Dia(w, f^{-1}(\xi))}}{e^{\Dia(z,f^{-1}(\xi))}} \right) e^{\Dia(z,f^{-1}(\xi))} \Bergman(\xi,\xi) dV(\xi) \\
		&= \int_{f^{-1}(\O)}  \chi^{(\alpha)} \left( \frac{e^{\Dia(w, \xi)}}{e^{\Dia(z,\xi)}} \right) e^{\Dia(z,\xi)} f^*\left( \Bergman(\xi,\xi) dV(\xi) \right) \\
		&= \int_{\O}  \chi^{(\alpha)} \left( \frac{e^{\Dia(w, \xi)}}{e^{\Dia(z,\xi)}} \right) e^{\Dia(z,\xi)} \Bergman(\xi,\xi) dV(\xi) \\
		&=  \Phi^* D^{(\alpha)}(z,w).
	\end{align*}
    Here, we used the facts that the diastasis function $\Dia(z,w)$ and the Bergman kernel form $\Bergman(\xi,\xi) dV(\xi)$ are invariant under biholomorphisms.
\end{proof}

        \begin{rmk}
            J. Burbea and C. Rao (\cite{BurbeaJacobRaoCRadhakrishna84}) proved that
            \begin{align*}
                (\Phi^*D^{(0)})(p,q) \le \rho(p,q)
            \end{align*}
            for all $p, q \in \O$,
            where $\rho$ is the Skwarczy\'nski distance (\cite{skwarczynski1980biholomorphic}) on $\O$.
        \end{rmk}

	
\vspace{5mm}  
	
	\section{Sufficient Statistic} \label{sec: sufficient statistic}
	
	\subsection{Information geometry}
	
	We first introduce the concept of sufficient statistic in the theory of information geometry based on \cite{ay2015information} and \cite{ay2017informationbook}.
	For two measurable spaces $\Xi_1$ and $\Xi_2$, a measurable map $f : \Xi_1 \rightarrow \Xi_2$ is called a {\it statistic}. 
	A statistic $f$ induces a map $\kappa : \Prob(\Xi_1) \rightarrow \Prob(\Xi_2)$ given by the measure push-forward by $f$, i.e., for $\mu \in \Prob(\Xi_1)$,
	\begin{equation} \label{def: kappa}
		\kappa(\mu)(B) := \mu(f^{-1}(B))
	\end{equation}
	for all Borel subsets $B \subset \Xi_2$. 
	Note that $\int_{\Xi_2} d\kappa(\mu) = \int_{\Xi_1} d\mu = 1$.
	For a $k$-integrable $(k \ge 1)$ statistical model $\Phi_1: M \rightarrow \Prob(\Xi_1, \mu_1)$, $\kappa$ induces a $k$-integrable statistical model $\kappa \circ \Phi_1: M \rightarrow \Prob(\Xi_2, \kappa(\mu_1))$ with the same parameter space $M$, where the density function is given by the Radon-Nikodym derivative
    \begin{align*}
        P_2(x, \zeta) := \frac{d((\kappa \circ \Phi_1)(x))}{d(\kappa(\mu_1))}(\zeta)
    \end{align*}
    for $x \in M$ and $\zeta \in \Xi_2$ (\cite[Theorem 5.4]{ay2017informationbook}).

    Assume that a statistical model $(M, \Xi_1, \mu_1, \Phi)$ is $2$-integrable. 
	\begin{defn} \label{def: sufficient statistic}
		A statistic $f: \Xi_1 \rightarrow \Xi_2$ is called {\it sufficient} for $\Phi: M \rightarrow \Prob(\Xi_1, \mu_1)$ if
		$$ ((\kappa \circ \Phi_1)^* g_{F_2})_{x}(V,V) = (\Phi_1^* g_{F_1})_{x}(V,V) $$
		for all $x \in M$ and $V \in T_x M$,
		where $g_{F_1}$ and $g_{F_2}$ are the Fisher information metrics of $\Prob(\Xi_1, \mu_1)$ and $\Prob(\Xi_2, \kappa(\mu_1))$, respectively. 
		Here, the pullback of $g_{F_2}$ is defined by 
		$$ ((\kappa \circ \Phi_1)^* g_{F_2})_{x}(V,V) := (g_{F_2})_{\kappa(\Phi_1(x))}(\kappa(d\Phi(V)), \kappa(d\Phi(V))) .$$
	\end{defn}
    
	In other words, sufficient statistics preserve the geometric structures (often interpreted as the amount of information in information geometry) of statistical models. Remarkably, the Fisher information metric and the Amari--Chentsov structure are uniquely determined by their invariance under sufficient statistics (\cite{ay2015information}).

    The terminology ``sufficient'' stems from its role in statistical inference. The primary objective of using statistical models is to infer the optimal parameter $x \in M$, representing the ``true'' probability distribution, based on observations from a sample space $\Xi_1$. If a statistic $f: \Xi_1 \rightarrow \Xi_2$ is sufficient, evaluating this statistic in $\Xi_2$ provides all the necessary information to identify the parameter.
    Several mathematical criteria exist to characterize such sufficient statistics.

	\begin{thm}[cf. {\cite[Theorem 2.1]{amari2000methods}, \cite[Lemma 3.3, Theorem 3.5]{ay2015information}, \cite[Proposition 5.5, Proposition 5.6]{ay2017informationbook}}]  
 \label{thm: equivalent conditions for sufficiency}
		Let $(M, \Xi_1, \mu_1, \Phi)$ be a $2$-integrable statistical model with a positive regular density function $P_1(x,\xi)$.
        Then the followings are equivalent.
		\begin{enumerate}
			\item A statistic $f: \Xi_1 \rightarrow \Xi_2$ is sufficient for $\Phi_1: M \rightarrow \Prob(\Xi_1, \mu_1)$.
			\item  The measurable function
			\begin{equation*}
				r(x, \xi) := \frac{P_1(x,\xi)}{P_2(x, f(\xi))} \quad \mu_1(\xi)\text{-a.e.}
			\end{equation*}
			does not depend on $x \in M$, where $\kappa(P_1(x,\xi)\mu_1(\xi)) = P_2(x, \zeta)\kappa(\mu_1)(\zeta)$.
			\item For each $x \in \O_1$, there exist functions $s(x, \cdot) \in L^1(\Xi_2, \kappa(\mu_1))$ and $t \in L^1(\Xi_1, \mu_1)$ such that
			\begin{equation*}
				P_1(x, \xi) = s(x, f(\xi)) t(\xi) \quad \mu_1(\xi)\text{-a.e.}
			\end{equation*} 
		\end{enumerate}
	\end{thm}

    For the proof, we refer readers to \cite{ay2017informationbook}. 
    In \cite{ay2017informationbook}, Proposition 5.5 and Proposition 5.6 show that $(3) \Leftrightarrow (2)$ and $(2) \Leftrightarrow (1)$, respectively. \\

	One of the most interesting features for the Fisher information metric is the following monotonic decreasing property. 

	\begin{thm}[cf. {\cite[Theorem 2.1]{amari2000methods}, \cite[Theorem 3.11]{ay2015information}, \cite[Theorem 5.4]{ay2017informationbook}  }] \label{thm: monotonicity Fisher metric} 
		For a statistic $f : \Xi_1 \rightarrow \Xi_2$, it holds that
		$$ ((\kappa \circ \Phi_1)^* g_{F_2})_{x}(V,V) \le (\Phi_1^* g_{F_1})_{x}(V,V) $$
		for all $x \in M$ and $V \in T_x M$.
	\end{thm}

\vspace{2mm}  	
	\subsection{Bergman statistical model}

	Let $(\O_1, g_{B_1})$ and $(\O_2, g_{B_2})$ be bounded domains in $\CC^n$ with the Bergman metrics and $dV$ be the Lebesgue measure on $\CC^n$.
	Let $\Bergman_1$ and $\Bergman_2$ be the Bergman kernels of $\O_1$ and $\O_2$, respectively.
	Let $\Phi_1 : \O_1 \rightarrow \Prob(\O_1, dV)$ be the Bergman statistical model defined by (\ref{def: map into P(D)}) and $f: \O_1 \rightarrow \O_2$ be a measurable function.
	In this situation, we have the following diagram.
	\begin{equation} \label{eq:diagram triangle}
		\begin{tikzcd}
			(\O_1, g_{B_1}) \arrow{r}{\Phi_1} \arrow[swap]{dr}{\kappa \circ \Phi_1 } &  \arrow{d}{\kappa}  (\Prob(\O_1, dV), g_{F_1})  \\
			& (\Prob(\O_2, \kappa(dV)), g_{F_2})
		\end{tikzcd}
	\end{equation}
	where $\kappa$ is defined by (\ref{def: kappa}). 
	Then Definition \ref{def: sufficient statistic} and Theorem \ref{thm: monotonicity Fisher metric} yield
	\begin{thm} \label{thm: decreasing property of Fisher metric applying to Bergman}
		For a measurable function $f : \O_1 \rightarrow \O_2$, 
		\begin{equation} \label{inequality monotonic decreasing}
			(\kappa \circ \Phi_1)^* g_{F_2} \le \Phi_1^* g_{F_1} = g_{B_1}
		\end{equation}
		on $\O_1$
		and the equality holds if and only if $f$ is sufficient.
	\end{thm}
	
	By Theorem \ref{thm: decreasing property of Fisher metric applying to Bergman}, we have a new Riemannian metric $(\kappa \circ \Phi_1)^* g_{F_2}$ on $\O_1$ depending on $f$ which is always less than or equal to the Bergman metric of $\O_1$.  
	Unfortunately, we do not know whether it is Hermitian or not, even though $f$ is a holomorphic function. \\

        For a $C^1$-smooth map $f: \O_1 \rightarrow \O_2$, 
        let $Z := \{ f(z) \in \O_2 : |J_{\RR}f (z)|=0 \}$ be the set of all critical values, 
        where $|J_{\RR}f (z)|$ is the determinant of the real Jacobian matrix of $f$.
        Note that $Z$ is a measure zero subset in $\O_2$ by Sard's theorem.
	Suppose that 
	\begin{align} \label{f locally invertible}
		&f: f^{-1}(\O_2 \setminus Z) \rightarrow (\O_2 \setminus Z) \text{ is a } m\text{-sheeted } C^1\text{-smooth covering map} 
	\end{align} 
	for some $m \in \NN$.
	In this case, $\kappa$ can be written as 
	\begin{equation} \label{kappa expression}
		\kappa(P_1(z, \xi) dV(\xi))(z, \zeta) = \sum_{k=1}^{m} \frac{|\Bergman_{1}(z, f^{-1}_k(\zeta) )|^2 |J_{\RR}f^{-1}_k(\zeta)|}{\Bergman_{1}(z,z)} dV(\zeta) 
		= \frac{K(z, \zeta)}{\Bergman_{1}(z,z)} dV(\zeta)
	\end{equation}
	for $\zeta \in \O_2 \setminus Z$,
	where $f^{-1}_k$ is a local inverse of $f$ and $K(z, \zeta) :=  \sum_{k=1}^{m} |\Bergman_{1}(z, f^{-1}_k(\zeta) )|^2 |J_{\RR}f^{-1}_k(\zeta)|$.  \\
	
	Now, assuming that $f$ satisfies (\ref{f locally invertible}), we will give a proof of the following inequality 
	\begin{align*}
		((\kappa \circ \Phi_1)^* g_{F_2})_z (\partial_{\alpha}, \partial_{\overline{\alpha}})
		\le (\Phi_1^* g_{F_1})_z (\partial_{\alpha}, \partial_{\overline{\alpha}})
	\end{align*}
	for all $z \in \O_1$ and $\partial_{\alpha} = \frac{\partial}{\partial z_{\alpha}} = \frac{1}{2}( \frac{\partial}{\partial x_{\alpha}} - \frac{\partial}{\partial y_{\alpha}}) \in T^{(1,0)}_z \O_1$,
	and induce a sufficient and necessary condition for the equality in terms of the Bergman kernels, which will be used to prove Theorem \ref{thm: sufficient iff 1-1}.
	Note that for a Riemannian metric $g$ (not necessarily Hermitian), 
	$$ 4 \hat{g} \left(\frac{\partial}{\partial z_{\alpha}}, \frac{\partial}{\partial \overline{z}_{\alpha}} \right) 
	=   g \left(\frac{\partial}{\partial x_{\alpha}},\frac{\partial}{\partial x_{\alpha}} \right) 
	+  g\left(\frac{\partial}{\partial y_{\alpha}}, \frac{\partial}{\partial y_{\alpha}} \right)  ,$$
	where $\hat{g}$ is the $\CC$-linearly extension of $g$.
	Recall that 
	$$ P_1(z,\xi) = \frac{|\Bergman_{1}(z, \xi)|^2}{\Bergman_{1}(z,z)}.
	$$

        \begin{lem} \label{lem: int K = B}
            For each $z \in \O_1$ and $1 \le \alpha \le n$, the following identities hold.
            \begin{align*}
                \int_{\O_2} \partial_{\alpha} K(z,\zeta) dV(\zeta) = \partial_{\alpha} \Bergman_1(z,z)
            \end{align*}
            and 
            \begin{align*}
                \int_{\O_2} \partial_{\alpha} \partial_{\overline{\alpha}} K(z,\zeta) dV(\zeta) = \partial_{\alpha} \partial_{\overline{\alpha}} \Bergman_1(z,z).
            \end{align*}
        \end{lem}
        \begin{proof}
            Fix $z \in \O_1$ and $1 \le \alpha \le n$. Let $\{ s_j \}_{j=0}^{\infty}$ be a special orthonormal basis with respect to $z$ and $\partial_{\alpha} \in T_z^{1,0}\O$ (Section~\ref{subsec: Special orthonormal basis}). Then
            \begin{align*}
                \partial_{\alpha} K(z,\zeta) 
                &= \sum_{k=1}^{m} \left( \partial_{\alpha} s_0(z) \overline{s_0(f^{-1}_k(\zeta))} + \partial_{\alpha}s_1(z) \overline{s_1(f^{-1}_k(\zeta))} \right)
                \left( \overline{s_0(z)} s_0(f^{-1}_k(\zeta)) \right)
                |J_{\RR}f^{-1}_k(\zeta)| \\
                &= \partial_{\alpha}s_0(z) \overline{s_0(z)} \sum_{k=1}^{m} |s_0(f^{-1}_k(\zeta))|^2
                |J_{\RR}f^{-1}_k(\zeta)| \\
                &+ \partial_{\alpha}s_1(z) \overline{s_0(z)} \sum_{k=1}^{m} s_0(f^{-1}_k(\zeta)) \overline{s_1(f^{-1}_k(\zeta))}
                |J_{\RR}f^{-1}_k(\zeta)|.
            \end{align*}
            Also, 
            \begin{align*}
                \int_{\O_2} \sum_{k=1}^{m} |s_0(f^{-1}_k(\zeta))|^2
                |J_{\RR}f^{-1}_k(\zeta)| dV(\zeta) = \int_{\O_1} |s_0(\xi)|^2 dV(\xi) = 1
            \end{align*}
            and
            \begin{align*}
                \int_{\O_2} \sum_{k=1}^{m} s_0(f^{-1}_k(\zeta)) \overline{s_1(f^{-1}_k(\zeta))}
                |J_{\RR}f^{-1}_k(\zeta)| dV(\zeta) = \int_{\O_1} s_0(\xi) \overline{s_1(\xi)} dV(\xi) = 0
            \end{align*}
            because $\{ s_j \}_{j=0}^{\infty}$ is an orthonormal basis.
            Therefore, 
            \begin{align*}
                \int_{\O_2} \partial_{\alpha} K(z,\zeta) dV(\zeta)
                = \partial_{\alpha}s_0(z) \overline{s_0(z)}
                = \partial_{\alpha} \Bergman_1(z,z).
            \end{align*}
            Similarly, 
            \begin{align*}
                \int_{\O_2} \partial_{\alpha} \partial_{\overline{\alpha}} K(z,\zeta) dV(\zeta)
                = |\partial_{\alpha}s_0(z)|^2 + |\partial_{\alpha}s_1(z)|^2
                = \partial_{\alpha} \partial_{\overline{\alpha}} \Bergman_1(z,z).
            \end{align*}
        \end{proof}

	For $z \in \O_1$,
	\begin{align*} 
		&((\kappa \circ \Phi_1)^* g_{F_2})_z (\partial_{\alpha}, \partial_{\overline{\alpha}}) \\
		=& \int_{\O_2} | \partial_{\alpha} \log \kappa(\Poisson_{1} dV) |^2 \kappa(\Poisson_{1} dV)  \\
		=& \int_{\O_2} \left( |\partial_{\alpha} \log K|^2 - \partial_{\alpha}\log K \cdot \partial_{\overline{\alpha}} \log \Bergman_1 - \partial_{\overline{\alpha}}\log K \cdot \partial_{\alpha} \log \Bergman_1 + |\partial_{\alpha} \log \Bergman_1|^2 \right) \frac{K}{\Bergman_1} dV \\
            =& \frac{1}{\Bergman_1} \int_{\O_2} \frac{|\partial_{\alpha} K|^2}{K} dV - |\partial_{\alpha} \log \Bergman_1|^2 \\
            =& \frac{1}{\Bergman_1} \int_{\O_2} \frac{|\partial_{\alpha} K|^2}{K} dV - \frac{\partial_{\alpha} \partial_{\overline{\alpha}} \Bergman_1}{\Bergman_1} + \partial_{\alpha}\partial_{\overline{\alpha}}\log \Bergman_1 \\
            =& \frac{1}{\Bergman_{1}} \int_{\O_2} \left( \frac{|\partial_{\alpha} K|^2}{K} - \partial_{\alpha}\partial_{\overline{\alpha}} K \right) dV 
		+ \partial_{\alpha}\partial_{\overline{\alpha}} \log \Bergman_{1},
	\end{align*}
        where Lemma~\ref{lem: int K = B} is applied in the third and fifth equalities.
	Therefore,
	\begin{align} \label{4.44}
		&((\kappa \circ \Phi_1)^* g_{F_2})_z (\partial_{\alpha}, \partial_{\overline{\alpha}}) \\
		=& \frac{1}{\Bergman_{1}(z,z)} \int_{\O_2} \left( \frac{|\partial_{\alpha} K(z, \zeta)|^2}{K(z, \zeta)} - \partial_{\alpha}\partial_{\overline{\alpha}} K(z, \zeta) \right) dV(\zeta) 
		+ \partial_{\alpha}\partial_{\overline{\alpha}} \log \Bergman_{1}(z,z). \nonumber
	\end{align}

	\begin{lem} \label{lem: K inequality} 
		For each $z \in \O_1$ and $1 \le \alpha \le n$, the following inequality always holds.
		\begin{equation} \label{inequality K(z, xi)}
			\frac{|\partial_{\alpha} K(z, \zeta)|^2}{K(z,\zeta)} \le \partial_{\alpha}\partial_{\overline{\alpha}} K(z, \zeta) 
		\end{equation}
		for all $\zeta \in \O_2$ and the equality holds if and only if 
		$$ \partial_{\alpha} \log \Bergman_{1}(z, f_1^{-1}(\zeta)) 
		= \cdots  
		= \partial_{\alpha} \log \Bergman_{1}(z, f_m^{-1}(\zeta)).
		$$
	\end{lem}
	\begin{proof}
		\begin{align*}
			K &= \sum_{k=1}^m |\Bergman_{1}(z, f_k^{-1}(\zeta))|^2 |J_{\RR}f_k^{-1}(\zeta)| ,\\
			\partial_{\alpha} K &= \sum_{k=1}^m \partial_{\alpha} \Bergman_{1}(z, f_k^{-1}(\zeta)) \Bergman_{1}(f_k^{-1}(\zeta), z) |J_{\RR}f_k^{-1}(\zeta)| ,\\
			\partial_{\alpha} \partial_{\overline{\alpha}} K &= \sum_{k=1}^m |\partial_{\alpha} \Bergman_{1}(z, f_k^{-1}(\zeta))|^2 |J_{\RR}f_k^{-1}(\zeta)| , \\ 
			|\partial_{\alpha} K|^2 &= 
			\left| \sum_{k=1}^m \left(  \Bergman_{1}(f_k^{-1}(\zeta), z) |J_{\RR}f_k^{-1}(\zeta)|^{\frac{1}{2}}  \right) 
			\left( \partial_{\alpha} \Bergman_{1}(z, f_k^{-1}(\zeta)) |J_{\RR}f_k^{-1}(\zeta)|^{\frac{1}{2}}   \right)
			\right|^2 \\
			&\le K \cdot \partial_{\alpha} \partial_{\overline{\alpha}} K,
		\end{align*}	    
		where, in the last inequality, the Cauchy-Schwarz inequality is applied.
		The equality holds if and only if 
		{\setlength\arraycolsep{2pt}
			$$
			\begin{array}{rrl}
				&  \partial_{\alpha} \Bergman_{1}(z, f_k^{-1}(\zeta)) |J_{\RR}f_k^{-1}(\zeta)|^{\frac{1}{2}}
				&= C(z, \zeta) \cdot \Bergman_{1}(z, f_k^{-1}(\zeta)) |J_{\RR}f_k^{-1}(\zeta)|^{\frac{1}{2}}    \\ [.5em]
				\Leftrightarrow \quad\quad &
				\partial_{\alpha} \log \Bergman_{1}(z, f_k^{-1}(\zeta)) &= C(z, \zeta)
			\end{array}
			$$
		}
		for all $k = 1, \cdots, m$.
	\end{proof}
	Therefore, Lemma \ref{lem: K inequality} implies that, in (\ref{4.44}), 
	\begin{align*}
		((\kappa \circ \Phi_1)^* g_{F_2})_z (\partial_{\alpha}, \partial_{\overline{\alpha}})
		\le \partial_{\alpha}\partial_{\overline{\alpha}} \log \Bergman_{1}(z,z) 
		= (\Phi_1^* g_{F_1})_z (\partial_{\alpha}, \partial_{\overline{\alpha}}).
	\end{align*}
	

	\begin{thm} \label{thm: sufficient iff 1-1}
		Assume that $f : \O_1 \rightarrow \O_2$ satisfies (\ref{f locally invertible}).
        Then $f$ is sufficient for the Bergman statistical model if and only if $f: f^{-1}(\O_2 \setminus Z) \rightarrow (\O_2 \setminus Z)$ is injective,
        where $Z = \{ f(z) \in \O_2 : |J_{\RR}f (z)|=0 \}$.
	\end{thm}
	\begin{proof}
        If $f: f^{-1}(\O_2 \setminus Z) \rightarrow (\O_2 \setminus Z)$ is injective, then there exists a measurable function $f^{-1} : \O_2 \rightarrow \O_1$ such that $f \circ f^{-1} = \text{id}$. This allows for the factorization $P_1(z,\xi) = s(z, f(\xi)) t(\xi)$, where $s(z, \zeta) = P_1(z, f^{-1}( \zeta ))$ and $t(\xi) = 1$.
        Consequently, Theorem~\ref{thm: equivalent conditions for sufficiency} implies $f$ is a sufficient statistic. \\
		
		Fix $p \in \O_1$. 
            Let $\{ s_j \}_{j=0}^{\infty}$ be a special orthonormal basis for $A^2(\O_1)$ with respect to $p$ constructed in Section~\ref{subsec: Special orthonormal basis}.
		Recall that 
            $
            D^1 = \frac{\partial}{\partial z_1}, 
            D^2 = \frac{\partial}{\partial z_2}, \cdots, 
		D^n = \frac{\partial}{\partial z_n}, 
            D^{n+1} = \frac{\partial^2}{\partial z_1 \partial z_1},  
		D^{n+2} = \frac{\partial^2}{\partial z_1 \partial z_2}, \cdots $ and so on.

        Assume that $f$ is sufficient and for the sake of contradiction $f: f^{-1}(\O_2 \setminus Z) \rightarrow (\O_2 \setminus Z)$ is not injective, i.e., there exist $p \neq q \in f^{-1}(\O_2 \setminus Z)$ with $f(p)=f(q)$. Then (\ref{4.44}) and Lemma \ref{lem: K inequality} imply that 
		\begin{equation} \label{4.55}
			D^{k} \log \Bergman_{1}(p, p) = D^{k} \log \Bergman_{1}(p, q)
		\end{equation}
		for all $k \ge 1$.  \\
		
		\noindent
		{\bf Claim}: $s_j(q)=0$ for all $j \ge 1$.
		\begin{proof}[Proof of the claim]
			We use the mathematical induction. With the special orthonormal basis $\{ s_j \}_{j=0}^{\infty}$,
			\begin{align*}
				\Bergman_{1}(p,q) &= s_0(p) \overline{s_0(q)},  \\
				D^1 \Bergman_{1}(p,q) &= D^1 s_0(p) \overline{s_0(q)} + D^1 s_1(p) \overline{s_1(q)}, \\
				D^1 \log \Bergman_{1}(p,q) &= \frac{D^1 s_0(p)}{s_0(p)} + \frac{D^1 s_1(p) \overline{s_1(q)}}{s_0(p) \overline{s_0(q)}}.
			\end{align*}
			Then the equation (\ref{4.55}) for $k=1$ implies that 
			$$ \frac{D^1 s_0(p)}{s_0(p)} + \frac{D^1 s_1(p) \overline{s_1(q)}}{s_0(p) \overline{s_0(q)}}
			= \frac{D^1 s_0(p)}{s_0(p)} ,$$
			which concludes that $s_1(q)=0$ because $D^1 s_1(p) \neq 0$.
			Now, for $N \ge 1$ suppose that $s_j(q) = 0$ for all $1 \le j \le N$. Then 
			\begin{align*}
				D^{N+1} \Bergman_{1}(p,q) 
				&= D^{N+1} s_0(p) \overline{s_0(q)} + D^{N+1} s_{N+1}(p) \overline{s_{N+1}(q)} \\
				D^{N+1} \log \Bergman_{1}(p,q) 
				&= \frac{D^{N+1} \Bergman_{1}(p,q)}{\Bergman_{1}(p,q)} + R(p,q)    \\
				&= \frac{D^{N+1} s_0(p)}{s_0(p)} + \frac{D^{N+1} s_{N+1}(p) \overline{s_{N+1}(q)}}{s_0(p) \overline{s_0(q)}} + R(p,q)
			\end{align*}
			where $R(p,q)$ is the remaining term. Note that $R(p,q)$ is a combination of $D^M \Bergman_{1}(p,q) / \Bergman_{1}(p,q)$ for $1 \le M \le N$, but then, from the assumption that $s_j(q) = 0$ for all $1 \le j \le N$,
			$$ \frac{D^M \Bergman_{1}(p,q)}{\Bergman_{1}(p,q)} = \frac{D^M s_0(p) \overline{s_0(q)}}{s_0(p) \overline{s_0(q)}} 
			= \frac{D^M s_0(p)}{s_0(p)},
			$$
			which depends only on $p$. Hence, $R(p,q) = R(p)$ depends only on $p$.
			Again the equation (\ref{4.55}) for $k = N+1$ yields that 
			$$ \frac{D^{N+1} s_0(p)}{s_0(p)} + \frac{D^{N+1} s_{N+1}(p) \overline{s_{N+1}(q)}}{s_0(p) \overline{s_0(q)}}
			= \frac{D^{N+1} s_0(p)}{s_0(p)}.$$
			Therefore, $s_{N+1}(q) = 0$ because $D^{N+1} s_{N+1}(p) \neq 0$.
		\end{proof}
		Now, by Lemma \ref{lem: injective condition}, since $\Phi_1: \O_1 \rightarrow \Prob(\O_1)$ is injective (Theorem~\ref{thm: Phi is an injective immersion}), there exists $\phi \in A^2(\O_1)$ such that $\phi(p) = 0$ and $\phi(q) \neq 0$, and it can be written as $\phi(z) = \sum_{j=0}^{\infty} a_j s_j(z)$ for some $a_j \in \CC$.
		Then 
		$$ 0 = \phi(p) = a_0 s_0(p) \quad \text{and} \quad 0 \neq \phi(q) = a_0 s_0(q) $$
		Since $s_0(p) \neq 0$, we have $a_0 = 0$ and hence $\phi(q) = 0$, which is a contradiction.	 \\
	\end{proof}

    Let $f: \O_1 \rightarrow \O_2$ be a proper holomorphic map. According to a result by R. Remmert, the set of critical values $Z = \{ f(z) \in \O_2 : |J_{\CC} f(z)|=0 \}$, where $|J_{\CC} f(z)|$ is the determinant of the complex Jacobian matrix of $f$, is a complex variety in $\O_2$. Under this assumption, $f$ satisfies the condition (\ref{f locally invertible}). 
    We note that $|J_{\CC} f(z)|^2 = |J_{\RR} f(z)|$, and $f$ is necessarily surjective (\cite[Proposition 15.1.5]{rudin2012function_book}). 
    By applying Theorem~\ref{thm: sufficient iff 1-1}, we obtain the following characterization for biholomorphisms.

\begin{cor} \label{cor: proper holo 1-1 if and only if sufficient}
	A proper holomorphic map $f: \O_1 \rightarrow \O_2$ is a biholomorphism if and only if $f$ is sufficient for the Bergman statistical model.
\end{cor}
\begin{proof}
	It follows from Theorem~\ref{thm: sufficient iff 1-1} and the fact that the number of preimage of singular value is less than that of preimage of regular value when $f$ is proper holomorphic (\cite[Theorem 15.1.9]{rudin2012function_book}).
\end{proof}

	Now, we consider the following diagram, which does not commute in general.
	\begin{equation}\label{eq:diagram rectangle}
		\begin{tikzcd}
			(\O_1, g_{B_1}) \arrow[swap]{d}{f} \arrow{r}{\Phi_1}  & \arrow{d}{\kappa}   (\Prob(\O_1), g_{F_1})  \\
			(\O_2, g_{B_2}) \arrow{r}{\Phi_2}  &  (\Prob(\O_2), g_{F_2})
		\end{tikzcd}
	\end{equation}
	
	\begin{prop} \label{prop: comparison btw P_1 and P_2}
		For $z \in \O_1$,
		$$ (\kappa \circ \Phi_1)(z) \ge \frac{\Bergman_2(f(z),f(z)) |J_{\CC}f(z)|^2 }{m \Bergman_1(z,z)} (\Phi_2 \circ f)(z).  $$
		where $J_{\CC}f(z)$ is the determinant of complex Jacobian matrix of $f$. Moreover, the equality holds if and only if
		\begin{equation} \label{4.10}
			\Bergman_1(z, f_1^{-1}(\zeta)) \overline{J_{\CC}f_1^{-1}(\zeta)} = \cdots = \Bergman_1(z, f_m^{-1}(\zeta)) \overline{J_{\CC}f_m^{-1}(\zeta)}
		\end{equation}
		for all $z \in \O_1$ and $\zeta \in \O_2 \setminus Z$.
	\end{prop}
	
	To prove it, we need the following transformation rule for the Bergman kernels by S. R. Bell.
	
	\begin{thm}[Theorem 1, \cite{bell1982bergman}] \label{thm: Bells transformation rule}
		Let $f: \O_1 \rightarrow \O_2$ be a proper holomorphic map. Then
		$$ J_{\CC}f(z) \Bergman_2(f(z), \zeta) = \sum_{k=1}^m \Bergman_1(z, f_k^{-1}(\zeta)) \overline{J_{\CC}f_k^{-1}(\zeta)}  $$
		for $z \in \O_1$ and $\zeta \in \O_2 \setminus V$.
	\end{thm}
	
	\begin{proof}[Proof of Proposition \ref{prop: comparison btw P_1 and P_2}]
		By Theorem \ref{thm: Bells transformation rule} and the following Cauchy-Schwarz lemma
		$$ \left|\sum_{k=1}^m a_k \right|^2 \le m \sum_{k=1}^m |a_k|^2 $$
		for $a_k \in \CC$, we have
		\begin{align} \label{4.1010}
			|J_{\CC}f(z)|^2 |\Bergman_2(f(z), \zeta)|^2 
			= \left| \sum_{k=1}^m \Bergman_1(z, f_k^{-1}(\zeta)) \overline{J_{\CC}f_k^{-1}(\zeta)} \right|^2
			\le m \sum_{k=1}^m |\Bergman_1(z, f_k^{-1}(\zeta))|^2 |J_{\CC}f_k^{-1}(\zeta)|^2 .
		\end{align}
		Then
		\begin{align*}
			(\kappa \circ \Phi_1)(z) 
			&= \kappa( P_1(z, \xi)dV(\xi)) \\
			&= \sum_{k=1}^m \frac{|\Bergman_1(z, f_k^{-1}(\zeta))|^2 |J_{\CC}f_k^{-1}(\zeta)|^2}{\Bergman_1(z,z)} dV(\zeta) \\
			&\ge \frac{ |\Bergman_2(f(z), \zeta)|^2 |J_{\CC}f(z)|^2 }{m \Bergman_1(z,z)} dV(\zeta) \\
			&= \frac{\Bergman_2(f(z),f(z)) |J_{\CC}f(z)|^2 }{m \Bergman_1(z,z)} (\Phi_2 \circ f)(z).
		\end{align*}
		From (\ref{4.1010}), the equality holds if and only if (\ref{4.10}).
	\end{proof}
	
		

	\begin{cor}
		Let $f : \O_1 \rightarrow \O_2$ be a proper holomorphic function. 
		If $f$ is injective, then the diagram (\ref{eq:diagram rectangle}) commutes, i.e., $\kappa \circ \Phi_1 = \Phi_2 \circ f$. 
		\begin{proof}
			If $f$ is injective, i.e., $m=1$, then the condition (\ref{4.10}) is satisfied and 
			$\Bergman_1(z,z) = \Bergman_2(f(z),f(z)) |J_{\CC}f(z)|^2  $ by the transformation rule.
			Hence, Proposition \ref{prop: comparison btw P_1 and P_2} implies that $\kappa \circ \Phi_1 = \Phi_2 \circ f.$
		\end{proof}
	\end{cor}
	
	\begin{rmk} \label{rmk: decreasing property of Bergman metric??}
		If the diagram (\ref{eq:diagram rectangle}) commutes, then $f^* g_{B_2} = (\Phi_2 \circ f)^* g_{F_2} = (\kappa \circ \Phi_1)^*g_{F_2} $.
		In this case, combining with Theorem \ref{thm: decreasing property of Fisher metric applying to Bergman}, we have the decreasing property for the Bergman metric, i.e., $f^* g_{B_2} \le  g_{B_1}$, which does not hold in general.  However, we do not know other sufficient conditions for commuting the diagram except for the injectivity of $f$.
	\end{rmk}

\vspace{5mm}
	
\section{Central limit theorem of Fr\'echet sample mean} \label{sec: asymptotic normality with the Bergman metric}

In this section, we prove consistency and the central limit theorem of the Fr\'echet sample mean replacing the squared geodesic distance with the diastasis function $\Dia$ for $g_B$ (Definition~\ref{def: diastasis function}).

    For a bounded domain $\Omega \subset \mathbb{C}^n$, let $\Phi: \O \rightarrow \Prob(\O, dV)$ be the Bergman statistical model defined by (\ref{def: map into P(D)}).
    For fixed $z_0 \in \Omega$, let $\{ Z_i \}_{i=1}^{\infty}$ be a sequence of independent and identically distributed (i.i.d.) random vectors drawn from the probability distribution 
    \[
    \Phi(z_0) = P(z_0, \xi) dV(\xi) = \frac{|\Bergman(z_0,\xi)|^2}{\Bergman(z_0,z_0)} dV(\xi).
    \]
	
	\begin{thm} [Existence and Consistency]  \label{thm: argmin converges a.s.} 
		Suppose that $\O$ satisfies Condition $C_6$ (Definition~\ref{def: Condition C_k}). Then, with probability tending to 1 as $m \rightarrow \infty$, there exists a sequence of random vectors
        \[
            \hat{z}_m := \arg \operatorname{loc} \min_{z \in \Omega} \sum_{i=1}^{m} \Dia(z, Z_i)
        \]
        such that  
		$$ \hat{z}_m  \xrightarrow{\quad p \quad} z_0 $$
		in probability as $m \rightarrow \infty$.
	\end{thm}
        \begin{rmk}
            The precise meaning for ``with probability tending to 1 as $m \rightarrow \infty$'' in Theorem~\ref{thm: argmin converges a.s.} is as follows:
            for a sequence of random vectors $Z_i$ from a probability space $(S, \mathcal{F}, P)$ to $\O$,
             \[
                \lim_{m \rightarrow \infty} P\left( \left\{ s \in S \,\,\middle\vert\,\, \operatorname{loc} \min_{z \in \Omega} \sum_{i=1}^{m} \Dia(z, Z_i(s))  \text{ exists } \right\} \right) = 1.
             \]
        \end{rmk}
    
	\begin{proof}[Proof of Theorem~\ref{thm: argmin converges a.s.}]
		We basically follow the argument of (\cite[Theorem 5.1 in Section 6]{LehmannELCasellaGeorge98}). 
		However, we do not assume the boundedness of third derivatives of $\log P(z, \xi)$. 
		Instead, we use the fact that 
		$\Exp[\partial_{\alpha}\partial_{\beta}\partial_{\overline{\gamma}} \log P(z_0,\xi)] = - \partial_{\beta} g_{\alpha \overline{\gamma}}(z_0)$ is finite (Lemma \ref{lem: 3-tensors identities}).
		Let
		\begin{equation*}
			L(z, Z) := L(z, Z_1, \cdots, Z_m) := \sum_{i=1}^m \Dia(z, Z_i).
		\end{equation*}
		Let $S_r$ be the (Euclidean) sphere of radius $r$ with the center at $z_0$. \\
		
		\noindent
		{\bf Claim: } $L(z, Z) > L(z_0, Z)$ on $S_r$ for sufficiently small $r>0$ with probability tending to 1.
		\begin{proof} [Proof of the claim]
			Consider the Taylor series expansion of $\frac{1}{m} L(z,Z) - \frac{1}{m} L(z_0, Z)$ near $z_0$ as follows.
			For $z \in S_r$,
			\begin{align*}
				\frac{1}{m} L(z,Z) - \frac{1}{m} L(z_0, Z)
				&= \frac{1}{m} \sum_{A} \frac{\partial L(z_0)}{\partial z_{A}} (z_{A} - z_{0, A}) \\
				&+ \frac{1}{2m} \sum_{A,B} \frac{\partial^2 L(z_0)}{\partial z_{A} \partial z_{B}} (z_{A} - z_{0, A}) (z_{B} - z_{0, B})  \\
				&+ \frac{1}{6m} \sum_{A,B,C} \frac{\partial^3 L(z^*)}{\partial z_{A} \partial z_{B} \partial z_{C}} (z_{A} - z_{0, A}) (z_{B} - z_{0, B}) (z_{C} - z_{0, C})   \\
				&=: Q_1 + Q_2 + Q_3.
			\end{align*}
			where $A, B, C = 1, \cdots, n, \overline{1}, \cdots, \overline{n}$, 
			$z = (z_1, \cdots, z_n)$, $z_0 = (z_{0,1}, \cdots, z_{0,n})$, $z_{\overline{\alpha}} = \overline{z}_{\alpha}$, $z_{0,\overline{\alpha}} = \overline{z}_{0,\alpha}$ for $\alpha = 1, \cdots, n$, and $z^*$ is on the line segment connecting $z$ and $z_0$. 
			Note that $\Dia(z, \xi) = - \log P(z,\xi) + \log \Bergman(\xi, \xi)$ and, by Lemma \ref{lem: 3-tensors identities} and \ref{lem: 4-tensors identities},  
			\begin{align*}
				\Exp\left[ \frac{\partial \Dia(z_0, Z_1)}{\partial z_A}  \right] 
				&= \int_{\O} \frac{\partial \Dia(z_0, \xi)}{\partial z_A}  P(z_0, \xi) dV(\xi) = 0, \\
				\Exp\left[ \frac{\partial^2 \Dia(z_0, Z_1)}{\partial z_A \partial z_B}  \right] 
				&= \int_{\O} \frac{\partial^2 \Dia(z_0, \xi)}{\partial z_A \partial z_B}  P(z_0, \xi) dV(\xi) = 
				\begin{cases}
					g_{\alpha \overline{\beta}}(z_0) \quad &(A=\alpha, B=\overline{\beta}),  \\
					0                                \quad &(\text{otherwise}),
				\end{cases} \\
				\Exp\left[ \frac{\partial^3 \Dia(z^*, Z_1)}{\partial z_A \partial z_B \partial z_C}  \right] 
				&= \int_{\O} \frac{\partial^3 \Dia(z^*, \xi)}{\partial z_A \partial z_B \partial z_C}  P(z_0, \xi) dV(\xi) 
				\rightarrow
				\begin{cases}
					0   \quad &(A=\alpha, B=\beta, C=\gamma), \\
					0   \quad &(A=\overline{\alpha}, B=\overline{\beta}, C=\overline{\gamma}), \\
					\frac{\partial}{\partial z_{\beta}} g_{\alpha \overline{\gamma}}(z_0)   \quad &(A=\alpha, B=\beta, C=\overline{\gamma}),
				\end{cases} 
			\end{align*}
			since $z^* \rightarrow z_0$ as $m \rightarrow \infty$, 
			where $g_B = \sum_{\alpha, \beta} g_{\alpha \overline{\beta}} dz_{\alpha} \otimes d\overline{z}_{\beta}  $ is the Bergman metric.
			Then by the law of large numbers (Theorem \ref{thm: Strong law})
			\begin{align}
				\frac{1}{m} \frac{\partial L(z_0)}{\partial z_A} &\longrightarrow  0,  \label{5.22} \\
				\frac{1}{m} \frac{\partial^2 L(z_0)}{\partial z_A \partial z_B} &\longrightarrow  
				\begin{cases}
					g_{\alpha \overline{\beta}}(z_0) \quad &(A=\alpha, B=\overline{\beta}), \label{5.33} \\
					0                                \quad &(\text{otherwise}),
				\end{cases} \\
				\frac{1}{m} \frac{\partial^3 L(z^*)}{\partial z_A \partial z_B \partial z_C} &\longrightarrow 
				\begin{cases}
					0   \quad &(A=\alpha, B=\beta, C=\gamma), \\
					0   \quad &(A=\overline{\alpha}, B=\overline{\beta}, C=\overline{\gamma}), \\
					\frac{\partial}{\partial z_{\beta}} g_{\alpha \overline{\gamma}}(z_0)   \quad &(A=\alpha, B=\beta, C=\overline{\gamma}),
				\end{cases}  \label{5.44}
			\end{align}
			in probability as $m \rightarrow \infty$.
			
			Now, for any given $r>0$, (\ref{5.22}) implies that $|Q_1| < a r^3$ for some $a>0$ with probability tending to 1.
			Also (\ref{5.33}) implies that $Q_2 > b r^2$ for some $b >0$ with probability tending to 1 because $g_{\alpha \overline{\beta}}(z_0) > 0$, and
			$|Q_3| < c r^3$ for some $c>0$ with probability tending to 1 follows from (\ref{5.44}) because $\frac{\partial}{\partial z_{\beta}} g_{\alpha \overline{\gamma}}(z_0)$ is finite.
			Therefore, 
			\begin{align*}
				\frac{1}{m} L(z,Z) - \frac{1}{m} L(z_0, Z) 
				&= Q_1 + Q_2 + Q_3 
				\ge -|Q_1| + Q_2 -|Q_3| \\
				&> -ar^3 + br^2 - cr^3 
				= br^2 - (a+c)r^3 > 0
			\end{align*}
			for sufficiently small $r>0$ with probability tending to 1.
		\end{proof}
		Since, for any $r>0$ small enough, $L(z, Z) > L(z_0, Z)$ for all $z \in S_r$ with probability tending to 1, we can choose a local minimum $\hat{z}_m$ of $L(z, Z)$ inside of $S_r$ such that $\hat{z}_m$ is measurable (see \cite[Problem 3.29 in Section 6]{LehmannELCasellaGeorge98} ) and 
        $\hat{z}_m \rightarrow z_0$ in probability as $m \rightarrow \infty$.
	\end{proof}
	
	\begin{thm}[Central Limit Theorem] \label{thm: CTL for Bergman metric}
            Suppose that $\O$ satisfies Condition $C_6$ (Definition~\ref{def: Condition C_k}). Then for the sequence $\{ \hat{z}_m \}_{m=1}^{\infty}$ in Theorem~\ref{thm: argmin converges a.s.},
    \[
    \sqrt{m} \left( \hat{z}_m - z_0 \right) \xrightarrow{\quad d \quad} \mathcal{N}_{\mathbb{C}}(0, g_B(z_0)^{-1},0)
    \]
    in distribution as $m \to \infty$, where the complex normal distribution $\mathcal{N}_{\CC}(0, g_B(z_0)^{-1},0)$ is defined by equation (\ref{def: complex normal distribution when R=0}).
	\end{thm}

	\begin{proof}
		Let
		\begin{equation*}
			L(z, Z) := L(z, Z_1, \cdots, Z_m) := \sum_{i=1}^m \Dia(z, Z_i).
		\end{equation*}
		For each $1 \le \beta \le n$, the first order Taylor series expansion of $\frac{\partial L}{\partial \overline{z}_{\beta}}$ at $z_0$ yields
		\begin{equation*}
			\frac{\partial L}{\partial \overline{z}_{\beta}}(z) 
			= \frac{\partial L}{\partial \overline{z}_{\beta}}(z_0) 
			+ \sum_{\alpha=1}^n \frac{\partial^2 L}{\partial z_{\alpha} \partial \overline{z}_{\beta}}(z^*) (z_{\alpha} - z_{0,\alpha}) 
			+ \sum_{\alpha=1}^n \frac{\partial^2 L}{\partial \overline{z}_{\alpha} \partial \overline{z}_{\beta}}(z^*) (\overline{z}_{\alpha} - \overline{z}_{0,\alpha}),
		\end{equation*}
		where $z = (z_1, \cdots, z_n)$, $z_0 = (z_{0,1}, \cdots, z_{0,n})$, and $z^*$ is on the line segment connecting $z$ and $z_0$. Since $\hat{z}_m  \rightarrow z_0$ by Theorem~\ref{thm: argmin converges a.s.}, we may let $z = \hat{z}_m$. Note that the left-hand side $\frac{\partial L}{\partial \overline{z}_{\beta}}(\hat{z}_m)=0$ because $L$ attains a local minimum value at $\hat{z}_m$. 
		Then, considering $\left( \hat{z}_m - z_0 \right)$ as a $(n \times 1)$ column vector,
		\begin{equation*}
			0 = \left( \frac{\partial L}{\partial \overline{z}_{\beta}}(z_0) \right)_{1 \le \beta \le n} 
			+ \left( \hat{z}_m - z_0 \right) \left( \frac{\partial^2 L(z^*)}{\partial z_{\alpha} \partial \overline{z}_{\beta}} \right)_{1 \le \alpha, \beta \le n}
			+ \left( \overline{\hat{z}}_m - \overline{z}_0 \right)  \left( \frac{\partial^2 L(z^*)}{\partial \overline{z}_{\alpha} \partial \overline{z}_{\beta}} \right)_{1 \le \alpha, \beta \le n}. 
		\end{equation*}
		Consequently, 
		\begin{align} \label{proof 4.2}
			\sqrt{m} \left( \hat{z}_m - z_0 \right)
			= &- \left( \frac{1}{\sqrt{m}} \frac{\partial L}{\partial \overline{z}_{\beta}}(z_0) \right)_{1 \le \beta \le n}
			\left( \frac{1}{m} \frac{\partial^2 L(z^*)}{\partial z_{\alpha} \partial \overline{z}_{\beta}} \right)^{-1}_{1 \le \alpha, \beta \le n} \\
			&- \sqrt{m} \left( \overline{\hat{z}}_m - \overline{z}_0 \right)  
			\left( \frac{1}{m} \frac{\partial^2 L(z^*)}{\partial \overline{z}_{\alpha} \partial \overline{z}_{\beta}} \right)_{1 \le \alpha, \beta \le n}
			\left( \frac{1}{m} \frac{\partial^2 L(z^*)}{\partial z_{\alpha} \partial \overline{z}_{\beta}} \right)^{-1}_{1 \le \alpha, \beta \le n}. \nonumber
		\end{align}
		Now, consider the following $\CC$-valued random matrices 
		$X_{i} := \left( \frac{\partial^2 \Dia(z^*, Z_i)}{\partial z_{\alpha} \partial \overline{z}_{\beta}} \right)_{\alpha\beta} $, 
		$Y_{i} := \left( \frac{\partial^2 \Dia(z^*, Z_i)}{\partial \overline{z}_{\alpha} \partial \overline{z}_{\beta}} \right)_{\alpha\beta}$ and
		$W_{i} := \left( \frac{\partial \Dia(z_0, Z_i)}{\partial \overline{z}_{\beta}} \right)_{\beta}$. Since
		\begin{align*}
			\Exp[X_i] &= \left( \int_{\O} \frac{\partial^2 \Dia(z^*, \xi)}{\partial z_{\alpha} \partial \overline{z}_{\beta}} P(z_0, \xi) dV(\xi)  \right)_{\alpha\beta} = g_B(z^*), \\
			\Exp[Y_i] &= \left( \int_{\O} \frac{\partial^2 \Dia(z^*, \xi)}{\partial \overline{z}_{\alpha} \partial \overline{z}_{\beta}} P(z_0, \xi) dV(\xi)  \right)_{\alpha\beta}
			= \left( - \int_{\O} \frac{\partial^2 \log \Poisson(z^*, \xi)}{\partial \overline{z}_{\alpha} \partial \overline{z}_{\beta}} P(z_0, \xi) dV(\xi)  \right)_{\alpha\beta}  \rightarrow {0},
		\end{align*}
		by (2) in Lemma~\ref{lem: 3-tensors identities} and $z^{*} \rightarrow z_0$, the strong law of large numbers (Theorem~\ref{thm: Strong law}) implies that 
		\begin{align} \label{proof 4.3}
			\left( \frac{1}{m} \frac{\partial^2 L(z^*)}{\partial z_{\alpha} \partial \overline{z}_{\beta}} \right)_{\alpha\beta} \xrightarrow{ d }  g_B(z_0)
			\quad \text{ and } \quad
			\left( \frac{1}{m} \frac{\partial^2 L(z^*)}{\partial \overline{z}_{\alpha} \partial \overline{z}_{\beta}} \right)_{\alpha\beta} \xrightarrow{ d } 0.
		\end{align}
		Moreover, 
		\begin{align*}
			\Exp[W_i] = \left( \int_{\O} \frac{\partial \Dia(z_0, \xi)}{\partial \overline{z}_{\beta}} P(z_0, \xi) dV(\xi)  \right)_{\beta}
			= \left( - \int_{\O} \frac{\partial \log \Poisson(z_0, \xi)}{\partial \overline{z}_{\beta}} P(z_0, \xi) dV(\xi)  \right)_{\beta} = 0,
		\end{align*}
		and the covariance matrix of $W_i$ and the relation matrix are
		\begin{align*}
			\text{Cor}[W_i, W_i] 
			&= \Exp[(W_i)(\overline{W_i})^{\top}] - \Exp[(W_i)] \overline{\Exp[(W_i)]}^{\top} \\
			&= \left( \int_{\O} \frac{\partial \Dia(z_0, \xi)}{\partial \overline{z}_{\beta}} \frac{\partial \Dia(z_0, \xi)}{\partial z_{\alpha}} P(z_0, \xi) dV(\xi)  \right)_{\beta\alpha} \\
			&= \left( \int_{\O} \frac{\partial \log \Poisson(z_0, \xi)}{\partial \overline{z}_{\beta}} \frac{\partial \log \Poisson(z_0, \xi)}{\partial z_{\alpha}} P(z_0, \xi) dV(\xi)  \right)_{\beta\alpha} \\
			&= \overline{g_B(z_0)}^{\top} = g_B(z_0) \\
			\text{Rel}[W_i, W_i]
			&= \Exp[(W_i)(W_i)^{\top}] - \Exp[(W_i)] \Exp[(W_i)]^{\top} = 0
		\end{align*}
		by Lemma~\ref{lem: 3-tensors identities}. Therefore, the Central Limit Theorem (Theorem~\ref{thm: Central limit thm}) implies that 
		\begin{equation} \label{proof 4.4}
			\left( \frac{1}{\sqrt{m}} \frac{\partial L}{\partial \overline{z}_{\beta}}(z_0) \right)_{\beta} 
			\xrightarrow{\quad d \quad} 
			\mathcal{N}_{\CC}(0, g_B(z_0),0). 
		\end{equation}
		Altogether (\ref{proof 4.2}), (\ref{proof 4.3}) and (\ref{proof 4.4}) with Slutsky's theorem (Theorem~\ref{thm: Slutshky}), we conclude that
		\begin{align*}
			\sqrt{m} \left(\hat{z}_m - z_0 \right) 	
			\xrightarrow{\quad d \quad} {g_B(z_0)}^{-1} \mathcal{N}_{\CC}(0, g_B(z_0),0) = \mathcal{N}_{\CC}(0, {g_B(z_0)}^{-1},0). 
		\end{align*}
		as $m \rightarrow \infty$.
	\end{proof}
    
	 As a consequence of Theorem~\ref{thm: argmin converges a.s.} and \ref{thm: CTL for Bergman metric}, we obtain the following corollary.
	\begin{cor} \label{cor: CLT for bounded symmetric domain}
		Let $\O \subset \CC^n$ be a bounded Hermitian symmetric domain.
		Then, with probability tending to 1 as $m \rightarrow \infty$, there exists a sequence random vectors
        \[
            \hat{z}_m := \arg \operatorname{loc} \min_{z \in \Omega} \sum_{i=1}^{m} \Dia(z, Z_i)
        \]
        such that  
		$$ \hat{z}_m  \xrightarrow{\quad p \quad} z_0 $$
		in probability and
		\begin{equation*}
			\sqrt{m}\left( \hat{z}_m - z_0  \right) \xrightarrow{\quad d \quad} \mathcal{N}_{\CC}(0, g_B(z_0)^{-1},0)
		\end{equation*}
		in distribution as $m \rightarrow \infty$.
	\end{cor} 
	\begin{proof}
		From Theorem~\ref{prop: symmetric domain satisfies Condition C_k}, $\O$ satisfies Condition $C_k$ for all $k \ge 1$.
	\end{proof}

\vspace{5mm}  
	
	\section{Appendix A. Information Geometry}  \label{sec: Appendix A}

	\subsection{$\alpha$-connection}

        Let $\Phi: M \rightarrow \Prob(\Xi, \mu_0)$ be a $3$-integrable statistical model (Definition~\ref{def: parametrized measure model}) with the density function 
        $P(x,\xi) = \frac{d\Phi(x)}{d\mu_0}$,
        and $g_F$ be the Fisher information metric (Definition~\ref{def: Fisher information metric}). 
	Denote by $l = l(x,\xi) := \log P(x,\xi)$, $\partial_{i} := \frac{\partial}{\partial x_{i}}$, where $( x_1, \cdots, x_n )$ is a coordinate system of $M$,
	and $\Exp_{x} [f] := \int_{\Xi} f(\xi) P(x, \xi) d\mu_0(\xi) $.\\
	
	 In general, $M$ naturally admits the following 1-parameter family of torsion free affine connections.
	
	\begin{defn}[$\alpha$-connection] \label{def: alpha connection}
		For $\alpha \in \RR$, the {\it $\alpha$-connection} $\nabla^{(\alpha)}$ on $M$ is an affine connection defined by 
		\begin{equation*}
			g_F \left( \nabla_{\partial_{i}}^{(\alpha)} \partial_{j}, \partial_{k}  \right)
			= \Gamma_{i j, k}^{(\alpha)},
		\end{equation*}
		where 
		\begin{align*}
			\left(\Gamma_{i j, k}^{(\alpha)}\right)_{x} 
			:= \Exp_{x} [ (\partial_{i} \partial_{j} l)(\partial_{k} l) ]
			+ \frac{1-\alpha}{2} \Exp_{x} [ (\partial_{i} l)(\partial_{j} l)(\partial_{k} l) ].
		\end{align*}
	\end{defn}
	
	\begin{rmk}
		$\nabla^{(0)}$ is the Levi-Civita connection with respect to $g_F$. In general, when $\alpha \neq 0$, $\nabla^{(\alpha)}$ is not the Levi-Civita connection. However, Theorem \ref{thm: pullback of Information--geometric concepts} says that $ \nabla^{(\alpha)}$ is the Levi-Civita connection for all $\alpha \in \RR$ for the Bergman statistical model.
	\end{rmk}

	\begin{defn}[Amari-Chentsov tensor] \label{def: Amari-Chentsov tensor}
		The {\it Amari-Chentsov tensor} $T=\left\{T_{i j k} \right\}$ is a totally symmetric $3$-tensor defined by
		$$
		(T_{i j k})_x :=  \Exp_x [(\partial_{i} l)(\partial_{j} l)(\partial_{k} l)].
		$$
	\end{defn}
	
	$T$ being totally symmetric means that $T_{i j k}=T_{\sigma(i) \sigma(j) \sigma(k)}$ for any permutation $\sigma$. 
	The relationship between the $\alpha$-connection and the $\beta$-connection is given by
	$$
	\Gamma_{i j, k}^{(\beta)}=\Gamma_{i j, k}^{(\alpha)}+\frac{\alpha-\beta}{2} T_{i j k}.
	$$
	We also have
	\begin{align*}
		\nabla^{(\alpha)} & =(1-\alpha) \nabla^{(0)}+\alpha \nabla^{(1)} \\
		& =\frac{1+\alpha}{2} \nabla^{(1)}+\frac{1-\alpha}{2} \nabla^{(-1)}.
	\end{align*}
	
	
	If a statistical model $\Phi: M \rightarrow \Prob(\Xi, \mu_0)$ with $\dim_{\RR} M = n$ can be expressed in terms of functions $\left\{C, F_{1}, \cdots, F_{n}\right\}$ on $\Xi$ and a function $\psi$ on $M$ as
	$$
	    P(x, \xi) = \exp \left[C(\xi)+\sum_{i=1}^{n} x_i F_{i}(\xi)-\psi(x)\right],
	$$
	then we say that $M$ is an {\it exponential family}. The map $\Phi$ is one-to-one if and only if the $n+1$ functions $\left\{F_{1}, \cdots, F_{n}, 1\right\}$ are linearly independent, where 1 denotes the constant function which identically takes the value 1.
	Many statistical models including the set of all normal distributions belong to an exponential family.

\vspace{2mm}  	
	\subsection{Divergence}
	
	Let $M$ be a smooth manifold.
	
	\begin{defn}[Divergence]
		A divergence on $M$ is a smooth function $D: M \times M \rightarrow \RR$ satisfying the following properties:
		\begin{enumerate}
			\item $D(p,q) \geq 0$ for all $p,q \in M$, and $D(p,q) = 0$ if and only if $p=q$,
			\item $g_{ij}^{(D)}(p) := - \frac{\partial}{\partial x_i} \frac{\partial}{\partial y_j} D(p,p) $  is positive definite for all $p \in M$,
		\end{enumerate}
		where $\frac{\partial}{\partial x_i}$ is applied to the first variable of $D$ while $\frac{\partial}{\partial y_i}$ to the second variable.
	\end{defn}
	Note that a divergence is not symmetric in general, i.e., $D(p,q) \neq D(q,p)$ and it follows from the condition (1) above that
	$$  \frac{\partial}{\partial x_i} D(p,p) =  \frac{\partial}{\partial y_i} D(p,p) = 0, $$
	$$  \frac{\partial^2}{\partial x_i \partial x_j} D(p,p) =  \frac{\partial^2}{\partial y_i \partial y_j} D(p,p) =  -\frac{\partial}{\partial x_i} \frac{\partial}{\partial y_j} D(p,p)   $$
	for all $p \in M$ because $D(x,p)$ attains a minimum at $x=p$.
	
	Given a divergence $D$, it induces a Riemannian metric $g_{ij}^{(D)}$ and an affine connection $\nabla^{(D)}$ whose the Christoffel symbols are given by
	$$ \Gamma_{i j, k}^{(D)}(p) := - \frac{\partial^2}{\partial x_i \partial x_j} \frac{\partial}{\partial y_k} D(p,p).  $$
	We have the following third order approximation of the divergence $D$:
	\begin{equation*}
		D(x,q) = \frac{1}{2} g_{ij}^{(D)}(q) \Delta x_i \Delta x_j + 
		\frac{1}{6} h_{ijk}^{(D)}(q) \Delta x_i \Delta x_j \Delta x_k + o(\norm{\Delta x}^3),
	\end{equation*}
	where 
	$$ h_{ijk}^{(D)}(q) := \frac{\partial^3}{\partial x_i \partial x_j \partial x_k} D(q,q) 
	= \frac{\partial}{\partial x_i} g_{jk}^{(D)}(q) + \Gamma_{jk,i}^{(D)}(q) .$$ \\

	Now, we introduce a 1-parameter family of divergences $D$ on a statistical model $\Phi: M \rightarrow \Prob(\Xi)$, called $\alpha$-divergence, such that $g^{(D)}$ and $\nabla^{(D)}$ are the Fisher information metric and the $\alpha$-connection (Definition \ref{def: alpha connection}), respectively.

	\begin{ex}[$\alpha$-divergence] \label{ex: alpha divergence}
		For $\alpha \in \RR$, let
		$$
		\chi^{(\alpha)}(x) = 
		\begin{cases}
			\frac{4}{1-\alpha^2} \left( 1 - x^{\frac{1+\alpha}{2}} \right) & (\alpha \neq \pm 1), \\
			x \log x & (\alpha = 1), \\
			- \log x & (\alpha = -1) .
		\end{cases}
		$$
		The {\it $\alpha$-divergence} $D^{(\alpha)}$ on $\Phi(M)$ is defined by
		$$
		D^{(\alpha)}(\mu_1 , \mu_2) := \int  \chi^{(\alpha)} \left(\frac{d \mu_2}{d \mu_1} \right)\mu_1 
		$$
		for $\mu_1, \mu_2 \in \Phi(M)$.
	\end{ex}
	$D^{(0)}$ and $D^{(\pm1)}$ are called the {\it Hellinger divergence} and {\it Kullback–Leibler (KL) divergence}, respectively.
	Note that $D^{(\alpha)}(\mu_1 , \mu_2) = D^{(-\alpha)}(\mu_2 , \mu_1)$ for all $\alpha \in \RR$.

	
\vspace{5mm}  	
	
	\section{Appendix B. Probability Theory with complex random variables} \label{sec: Appendix B}
	
	\subsection{Complex Random Variables and Complex Random Vectors}
	
	A probability space is a measure space $(S, \mathcal{F}, P)$ with a total measure one. Here, $S$ is usually called a sample space, $\mathcal{F}$ is a $\sigma$-algebra of subsets of $S$ where each element in $\mathcal{F}$ is called an event, and $P : \mathcal{F}  \rightarrow [0,1]$ is a probability measure. A {\it complex random variable} $X: S \rightarrow \mathbb{C}$ is a measurable function. 
	
	Given a complex random variable $Z$, it induces a probability distribution $P_Z$ on $\CC$ defined by 
	$$ P_Z(B) := P(Z^{-1}(B)) $$
	for any Borel set $B \subset \CC$.
	Then, by the Radon-Nikodym theorem, there exists a measurable function $f_Z$, called the {\it density function} of $Z$, such that 
	$$ P_Z(B) = \int_{B} f_Z(z) dV(z),$$
	where $dV(z)$ is the Lebesgue measure. 
	
	\begin{defn}
		Let $Z$ and $W$ be complex random variables. 
		\begin{enumerate}
			\item The {\it expectation} $\operatorname{E}[Z] := \int_{\CC} z f_Z(z) dV(z)$.
			\item The {\it covariance}  $\operatorname{Cov}[Z,W] := \operatorname{E} [(Z-\operatorname {E} [Z]){\overline {(W-\operatorname {E} [W])}}]=\operatorname {E} [Z{\overline {W}}]-\operatorname {E} [Z]\operatorname {E} [{\overline {W}}]$.
			\item The {\it variance} $\operatorname{Var}[Z]:=\operatorname {Cov} [Z,Z]=\operatorname {E} \left[\left|Z-\operatorname {E} [Z]\right|^{2}\right]=\operatorname {E} [|Z|^{2}]-\left|\operatorname {E} [Z]\right|^{2}.$
		\end{enumerate}
	\end{defn}
	If $\phi: \CC \rightarrow \CC$ is a measurable function, $\operatorname{E}[\phi(Z)] = \int_{\CC} \phi(z) f_Z(z) dV(z)$.
	An $n$-tuple of complex random variables $Z = (Z_1, \cdots, Z_n)$ is called an $n$-dimensional {\it complex random vector}.
	
	\begin{defn} \label{def: covariance}
		Let $Z=(Z_1, \cdots, Z_n)$ and $W$ be $n$-dimensional complex random vectors. 
		\begin{enumerate}
			\item The {\it expectation vector} $\operatorname{E}[Z] := (\operatorname{E}[Z_1], \cdots, \operatorname{E}[Z_n]) $.
			\item The {\it covariance matrix}  $\operatorname{Cov}[Z,W] := \operatorname{E}\left[(Z- \operatorname{E}[Z])(W- \operatorname{E}[W])^{\mathcal{H}}\right]$.
			\item The {\it relation matrix} $\operatorname{Rel}[Z,W]:= \operatorname{E}
   \left[(Z- \operatorname{E}[Z])(W- \Exp[W])^{\top}\right]$,
		\end{enumerate}
		where $A^{\top}$ is the transpose of $A$ and $A^{\mathcal{H}}$ is the conjugate transpose of $A$. 
	\end{defn}

	\begin{defn} 
		Let $Z, W$ be a complex random variables.
		\begin{enumerate}
			\item $Z, W$ are {\it independent} if $P(X \in A, Y \in B)=P(X \in A) P(Y \in B)$ for any Borel set $A, B \subset \CC$.
			\item $Z, W$ are {\it identically distributed} if $f_Z = f_W$.
			\item $Z, W$ are called {\it i.i.d.} if they are independent and identically distributed. 
		\end{enumerate}
	\end{defn}
	
	If $Z$ and $W$ are i.i.d., then $\operatorname{E}[Z] = \operatorname{E}[W]$, $\operatorname{Cov}[Z,Z] = \operatorname{Cov}[W,W]$ and $\operatorname{Rel}[Z,Z] = \operatorname{Rel}[W,W]$.

	\begin{defn}
		Let $\{ Z_i \}_{i=1}^{\infty}$ and $Z$ be a sequence of complex random variables and a complex random variable, respectively. 
		\begin{enumerate}
			\item $Z_i \xrightarrow{a.s.} Z$ {\it almost surely} if $Z_i$ converges to $Z$ almost everywhere.
			\item $Z_i \xrightarrow{p} Z$ {\it in probability} if $\lim_{i \rightarrow \infty}P\left( \{ s \in S : |Z_i(s)-Z(s)|\geq \epsilon\} \right)=0$ for all $\epsilon>0$.
			\item $Z_i \xrightarrow{d} Z$ {\it in distribution} if $P_{Z_i}(B)$ converges to $P_{Z}(B)$ for every Borel set $B \subset \CC$.
		\end{enumerate}
	\end{defn}
	
	Note that $(1) \Rightarrow (2) \Rightarrow (3)$. 
	Also, the convergence in probability implies the convergence on a subsequence almost surely.

\vspace{2mm}  	
	\subsection{Complex Normal Distributions}
	The multivariate complex normal distribution is introduced by Wooding \cite{Wooding1956}.
	For a complex $(n \times 1)$ column vector $z = (z_1, \cdots, z_n) \in \CC^n$, denote by $\hat{z} := (z_1, \cdots, z_n, \overline{z}_1, \cdots, \overline{z}_n) \in \CC^{2n}$.
	The multivariate {\it complex normal distribution} on $\CC^n$, denoted by $\mathcal{N}_{\CC}(\mu, \Gamma, R)$, is defined by the following density function  
	\begin{equation} \label{eq-4}
		f(z)=\frac{1}{\pi^{n}\left|\Gamma_{R}\right|^{1 / 2}} \exp \left\{-\frac{1}{2}(\hat{z}-\hat{\mu})^{\mathcal{H}} (\Gamma_{R})^{-1}(\hat{z}-\hat{\mu})\right\},
	\end{equation}
	where $\mu \in \mathbb{C}^{n}, \Gamma$ is a $(n \times n)$ Hermitian matrix, $R$ is a $(n \times n)$ symmetric matrix, 
	\begin{equation*}
		\Gamma_{R} :=\left(\begin{array}{cc}
			\Gamma & R \\
			R^{\mathcal{H}} & \Gamma^{*}
		\end{array}\right)=\left(\begin{array}{cc}
			\Gamma & R \\
			R^{*} & \Gamma^{*}
		\end{array}\right)
	\end{equation*}
	is a $(2n \times 2n)$ positive definite matrix and $\left|\Gamma_{R}\right|$ is the determinant of $\Gamma_R$.
	Here, $R^*$ and $R^{\mathcal{H}}$ denote the complex conjugate and the conjugate transpose of $R$, respectively. 
	When $R=0$, ~\eqref{eq-4} reduces to
	\begin{equation} \label{def: complex normal distribution when R=0} 
		f(z)=\frac{1}{\pi^{n}|\Gamma|} \exp \left\{-(z-\mu)^{\mathcal{H}} \Gamma^{-1}(z-\mu)\right\}.
	\end{equation}	
	
	A complex random vector $Z: S \rightarrow \CC^n$ is said to have a $n$-dimensional complex normal distribution, denoted by  $Z \sim  \mathcal{N}_{\CC}(\mu, \Gamma, R)$, if its density function $f_{Z}(z)$ is of the form (\ref{eq-4}), 
	where $\mu=\Exp[Z]$, $\Gamma={\text{Cov}}[Z,Z]=\Exp\left[(Z-\mu)(Z-\mu)^{\mathcal{H}}\right]$ and $R=\Exp\left[(Z-\mu)(Z-\mu)^{\top}\right]$. 
	Note that
	\begin{equation*}
		{\displaystyle Z\ \sim \ {\mathcal{N}}_{\CC}(\mu ,\,\Gamma ,\,R)\quad \Rightarrow \quad AZ+b\ \sim \ {\mathcal{N}}_{\CC}(A\mu +b,\,A\Gamma A^{\mathrm {H} },\,ARA^{\mathrm {T} })}
	\end{equation*}
	for a constant matrix $A$ and a vector $b$. 

 For completeness, we state the Central Limit Theorem and the Slutsky's theorem. See also \cite{ducharme2016complex}.
	
	\begin{thm}[Central Limit Theorem with complex normal distributions] \label{thm: Central limit thm}
		Let $ \{ Z_{i} \}_{i=1}^{\infty}$ be a sequence of i.i.d. $n$-dimensional complex random vectors with 
		the expectation vector $\mu = \operatorname{E}[Z_1] < \infty$, 
		the covariance matrix $\Gamma = \operatorname{Cor}[Z_1, Z_1] < \infty $ and 
		the relation matrix $R = \operatorname{Rel}[Z_1, Z_1] < \infty$.  Then
		$$
		\sqrt{m}\left(\frac{1}{m} \sum_{i=1}^{m} Z_{i} -\mu \right) \xrightarrow{d}  
		N_{\CC}\left(\mu, \Gamma, R \right)
		$$
		as $m \rightarrow \infty$.
	\end{thm}

	\begin{thm}[Slutsky's Theorem] \label{thm: Slutshky}
		Let  $\{ Z_{i} \}_{i=1}^{\infty}$ and  $\{ W_{i} \}_{i=1}^{\infty}$ be sequence of complex random vectors (resp. matrices). 
		If $Z_i \xrightarrow{d} Z$ and $W_i \xrightarrow{d} C$, where $Z$ is a complex random vector (resp. matrix) and $C$ is a constant vector (resp. matrix), then
		\begin{enumerate}
			\item $Z_i + W_i \xrightarrow{d} Z + C$.
			\item $Z_i W_i \xrightarrow{d} Z  C$.
			\item $Z_i / W_i \xrightarrow{d} Z / C$ provided that $C$ is invertible.
		\end{enumerate}
	\end{thm}

	\begin{thm}[Strong law of large numbers over complex numbers] \label{thm: Strong law}
		Let $\{ Z_{i} \}_{i=1}^{\infty}$ be a sequence of i.i.d. complex random vectors with $\operatorname{E}[|Z_1|] < \infty$.
		Then 
		$$ \frac{1}{m}\sum_{i=1}^{m}Z_i \xrightarrow{a.s.} \operatorname{E}[Z_1] $$ 
		as $m \rightarrow \infty$.
	\end{thm}

	\bibliographystyle{abbrv}
	\bibliography{reference}

\end{document}